\documentclass[a4paper,11pt]{article}                                                
\usepackage{graphicx}
\usepackage{amsmath}
\usepackage{amsfonts}
\usepackage{mathrsfs}
\usepackage{amsfonts,amsmath,amssymb,amsthm,enumerate,bm,a4}
\usepackage{mathtools}
\usepackage{hyperref}
\usepackage[a4paper,text={5.6in,8.25in},centering]{geometry}
\usepackage{authblk}

\theoremstyle{plain}
\newtheorem{theorem}{\noindent\bf Theorem}[section]

\theoremstyle{remark}

\newtheorem{remark}[theorem]{\noindent\bf Remark}
\numberwithin{equation}{section}

\newcommand{\bc}{\begin{center}}
\newcommand{\ec}{\end{center}}

\newcommand{\ba}{\begin{array}}
\newcommand{\ea}{\end{array}}
\newcommand{\beqn}{\begin{eqnarray}}
\newcommand{\eeqn}{\end{eqnarray}}
\newcommand{\be}{\begin{equation}}
\newcommand{\ee}{\end{equation}}
\newcommand{\ben}{\begin{enumerate}}
\newcommand{\een}{\end{enumerate}}

\allowdisplaybreaks[4]

\title{\bf On Bilinear Hardy Inequality and Corresponding Geometric Mean Inequality}
\author{Amiran Gogatishvili, Pankaj Jain and Saikat Kanjilal}
\date{}

\begin{document}
 \maketitle
 \noindent\textbf{Abstract.}
The main aim of this paper to provide several scales of equivalent conditions for the bilinear Hardy inequalities in the case $1< q, p_1, p_2<\infty$ with $q \geq \max(p_1,p_2)$.



\section{Introduction}

Let $\mathfrak{M}$ denote the set of all Lebesgue measurable functions  on $(a,b), \, -\infty\leq a < b\leq \infty$,
$\mathfrak{M}^{+}\subset\mathfrak{M}$ is the subset of all non-negative functions.

Let $ u,v, \in \mathfrak{M}^{+},  \;  0<p, q\leq \infty,\, p \geq 1 $.
Denote $p':=\frac{p}{p-1}$ and write
\begin{equation}
\label{e1}
U(x)= \int_{x}^{b} u(t)\, dt,\quad  V(x)= \int_{a}^{x}v^{1-p'}(t)\, dt,
\end{equation}

and assume that $U(x)< \infty, \, V(x)< \infty$ for almost everywhere (a.e.) $x \in (a,b)$.

Consider the one dimensional Hardy inequality
\begin{equation}
\label{H1}
\left(\int_a^b\left( \int_a^xf(t)\, dt \right)^q
u(x)\, dx\right)^\frac{1}{q}\leq C\left(\int_a^b
f^{p}(x)v(x)\, dx\right)^\frac{1}{p},  \quad f\in \mathfrak{M}^{+}.
\end{equation}
It is known that the inequality \eqref{H1} is characterized by the Muckenhoupt condition \cite{Mucken} in the case  $1<p\leq q<\infty$ which is given by
\begin{equation}
\label{MC}
A_M:= \sup_{a < x < b} A_M(x)= \sup_{a<x<b}U^\frac{1}{q}(x)V^\frac{1}{p'}(x)<\infty.
\end{equation}
It is noted that the Muckenhoupt condition $A_M < \infty$ is not unique. It can be replaced by several scales of equivalent conditions (see \cite{GKPW}, \cite{GKP}). Preciously, the following result was proved in \cite{GKP}:
\medskip

\textbf{Theorem A.} {\it Let $-\infty \leq a < b \leq \infty$, $\alpha, \beta,s$ be positive numbers and $f,g, h$ $\in \mathfrak{M}^{+}$. Denote
\begin{align} \label{e1.4}
F(x):=  \int_{x}^{b}f(t)\, dt, \quad G(x):=  \int_{a}^{x}g(t)\, dt,
\end{align}
and suppose that $F(x) < \infty$, $G(x) < \infty$ for every $x \in (a,b)$. Furthermore, denote
\allowdisplaybreaks[4]
\begin{align*}
&B_{1}(x; \alpha, \beta):= F^{\alpha}(x)G^{\beta}(x);\\
&B_{2}(x; \alpha, \beta, s):= \left({\int_{x}^{b}f(t)G^{\frac{\beta-s}{\alpha}}(t)\, dt}\right)^{\alpha}G^{s}(x);\\
&B_{3}(x; \alpha, \beta, s):= \left({\int_{a}^{x}g(t)F^{\frac{\alpha-s}{\beta}}(t)\, dt}\right)^{\beta}F^{s}(x);\\
&B_{4}(x; \alpha, \beta, s):= \left(\int_{a}^{x}f(t)G^{\frac{\beta+s}{\alpha}}\, dt\right)^{\alpha}G^{-s}(x); \\
&B_{5}(x; \alpha, \beta, s):= \left({\int_{x}^{b}g(t)F^{\frac{\alpha+s}{\beta}}(t)\, dt}\right)^{\beta}F^{-s}(x);\\
&B_{6}(x; \alpha, \beta, s):= \left(\int_{x}^{b}f(t)G^{\frac{\beta}{\alpha+s}}(t)\, dt\right)^{\alpha+s}F^{-s}(x);\\
&B_{7}(x; \alpha, \beta, s):= \left(\int_{a}^{x}g(t)F^{\frac{\alpha}{\beta+s}}(t)\, dt\right)^{\beta+s}G^{-s}(x); \\
&B_{8}(x; \alpha, \beta, s):= \left(\int_{a}^{x}f(t)G^{\frac{\beta}{\alpha-s}}(t)\, dt\right)^{\alpha-s}F^{s}(x), \quad \alpha > s; \\
&B_{9}(x; \alpha, \beta, s):= \left(\int_{x}^{b}f(t)G^{\frac{\beta}{\alpha-s}}(t)\, dt\right)^{\alpha-s}F^{s}(x), \quad \alpha < s; \\
&B_{10}(x; \alpha, \beta, s):= \left(\int_{x}^{b}g(t)F^{\frac{\alpha}{\beta-s}}(t)\, dt\right)^{\beta-s}G^{s}(x),\quad \beta>s; \\
&B_{11}(x; \alpha, \beta, s):= \left(\int_{a}^{x}g(t)F^{\frac{\alpha}{\beta-s}}(t)\, dt\right)^{\beta-s}G^{s}(x),\quad \beta<s; \\
&B_{12}(x; \alpha, \beta, s; h):= \left(\int_{x}^{b}f(t)h^{\frac{\beta-s}{\alpha}}(t)\, dt\right)^{\alpha}\Big(h(x)+G(x)\Big)^{s},\quad \beta<s; \\
& B_{13}(x; \alpha, \beta, s; h):= \left(\int_{a}^{x}g(t)h^{\frac{\alpha-s}{\beta}}(t)\, dt\right)^{\beta}\Big(h(x)+F(x)\Big)^{s}, \quad \alpha<s;  \\
& B_{14}(x; \alpha, \beta, s; h):= \left(\int_{a}^{x}f(t)\Big(h(t)+G(t)\Big)^{\frac{\beta+s}{\alpha}}\, dt\right)^{\alpha}h^{-s}(x); \\
& B_{15}(x; \alpha, \beta, s; h):= \left(\int_{x}^{b}g(t)\Big(h(t)+F(t)\Big)^{\frac{\alpha+s}{\beta}}\, dt\right)^{\beta}h^{-s}(x).
\end{align*}
The numbers $\displaystyle B_{1}(\alpha, \beta):= \sup_{a < x < b}B_{1}(x; \alpha, \beta)$, $\displaystyle B_{i}(\alpha, \beta, s):= \sup_{a < x < b} B_{i}(x; \alpha, \beta, s)$ $(i= 2, 3, ... , 11)$ and $\displaystyle B_{i}(\alpha, \beta, s):= \inf_{h \geq 0}\sup_{a < x < b} B_{i}(x; \alpha, \beta, s; h)$ $(i= 12, 13, 14, 15)$ are mutually equivalent.
}

For $\alpha= \frac{1}{q}$, $\beta= \frac{1}{p'}$, $\displaystyle F(x)= \int_{x}^{b} u(t)\, dt$ and $\displaystyle G(x)= \int_{a}^{x}v^{1-p'}(t)\, dt$, we find that
\begin{equation*}
B_1\left(x; \frac{1}{q}, \frac{1}{p'}\right)= A_M(x)
\end{equation*}
so that the condition
\begin{equation*}
\sup_{a<x<b}B_1\left(x; \frac{1}{q}, \frac{1}{p'}\right) < \infty
\end{equation*}
in Theorem A is the Muckenhoupt condition \eqref{MC}. Consequently, the other conditions in Theorem A are equivalent to \eqref{MC}. Let us mention that similar equivalent conditions for the inequality \eqref{H1} for the case $1<q<p<\infty$ were obtained by Persson, Stepanov and Wall \cite{PSW}.

Towards the first aim of this paper, we provide a stronger version of Theorem A by proving that, in Theorem A, the supremum over the interval $(a,b)$ can be considered in certain truncated intervals. This is done in Section 2.

Recently, Ca\~{n}estro \textit{et} \textit{al}. \cite{Agui} (see also \cite{Kr1}) considered the weighted bilinear Hardy inequality
\begin{align}
\label{H2}
\left(\int_a^b\left( \int_a^xf(t)\, dt\right)^q
\left( \int_a^xg(t)\, dt\right)^q
u(x)dx\right)^\frac{1}{q}&\leq C\left(\int_a^b
f^{p_1}(x)v_1(x)\, dt\right)^\frac{1}{p_1} \times \nonumber \\
& \quad \quad \times \left(\int_a^b
g^{p_2}(x)v_2(x)\right)^\frac{1}{p_2}
\end{align}
and proved the following:
\medskip

\textbf{Theorem B.} \emph{Let $1< q, p_1, p_2<\infty$ with $q \geq \max(p_1,p_2)$. Let $u, v_1, v_2$ $\in \mathfrak{M}^{+}$. Then there exits a positive constant $C$ such that the inequality \eqref{H2} holds for all $f, g \in \mathfrak{M}^{+}$ if and only if $\mathcal{D} < \infty$, where}
\begin{equation}
\label{BHC}
\mathcal{D}:= \sup_{a<x<b}\left(\int_{x}^{b} u(t)\, dt\right)^{\frac{1}{q}}\left(\int_{a}^{x}v_1^{1-p_1'}(t)\, dt\right)^{\frac{1}{p_1'}}\left(\int_{a}^{x}v_2^{1-p_2'}(t)\, dt\right)^{\frac{1}{p_2'}}.
\end{equation}

The next aim of this paper is to prove a result of the type of Theorem A in respect of the bilinear Hardy inequality \eqref{H2}. Some of these equivalent conditions have recently been proved in \cite{KPS}. In this paper, we give different proofs of some of the conditions proved in \cite{KPS} and moreover, we provide several new equivalent conditions. This is done in Section 3.

Finally, in Section 4, we give a characterization for the bilinear type geometric mean inequality

\begin{align*}
	\label{BGM}
	\left(\int_{0}^{\infty}\left(Tf\right)^q(x)\left(Tg\right)^q(x)u(x)\, dx\right)^{\frac{1}{q}} &\leq C\left(\int_{0}^{\infty}f^{p_1}(x)v_1(x)\, dx\right)^{\frac{1}{p_1}}\nonumber\\
	&\qquad\times\left(\int_{0}^{\infty}f^{p_2}(x)v_2(x)\, dx\right)^{\frac{1}{p_2}},
\end{align*}
where
\begin{align*}
	(Tf)(x):= \exp\left(\frac{1}{x}\int_{0}^{x} \ln(f(t))\, dt\right), \quad f \in \mathfrak{M}^{+}.
\end{align*}

\section{Improvement of the Theorem A}

We prove the following theorem:

\begin{theorem} \label{IThA}
Under the setting of Theorem A, for $x\in(a,b)$, the following hold:
\begin{itemize}
\item[(i)] $\displaystyle \sup_{x<y<b} B_1(y;\alpha, \beta) \approx \sup_{x<y<b} B_2(y;\alpha, \beta, s); $
\item[(ii)] $\displaystyle \sup_{x<y<b} B_1(y;\alpha, \beta) \approx \sup_{x<y<b}  \left(\int_{x}^{y}f(t)G^{\frac{\beta+s}{\alpha}}(t)\, dt\right)^{\alpha}G^{-s}(y); $
\item[(iii)] $\displaystyle \sup_{x<y<b} B_1(y;\alpha, \beta) \approx \sup_{x<y<b} B_6(y;\alpha, \beta, s); $
\item[(iv)] $\displaystyle \sup_{x<y<b} B_1(y;\alpha, \beta) \approx \sup_{x<y<b}  \left(\int_{x}^{y}f(t)G^{\frac{\beta}{\alpha-s}}(t)\, dt\right)^{\alpha-s}F^{s}(y) , \hskip 0.5em  \alpha>s; $
\item[(v)] $\displaystyle \sup_{x<y<b} B_1(y;\alpha, \beta) \approx \sup_{x<y<b} B_9(y;\alpha, \beta, s), \hskip 0.5em \alpha<s; $
\item[(vi)] $\displaystyle \sup_{x<y<b} B_1(y;\alpha, \beta) \approx \inf_{h \geq 0}\sup_{x<y<b} B_{12}(y;\alpha, \beta, s; h), \hskip 0.5em \beta<s; $
\item[(vii)] $\displaystyle \sup_{x<y<b} B_1(y;\alpha, \beta) \approx \inf_{h \geq 0} \sup_{x<y<b}h(y) ^{-s} \left(\int_x^y f(z)\Big(h(z)+G(z)\Big)^{\frac{\beta+s}{\alpha}}dz\right)^{\alpha}; $
\item[(viii)] $\displaystyle \sup_{a<y<x} B_1(y;\alpha, \beta) \approx \sup_{a<y<x} B_3(y;\alpha, \beta, s); $
\item[(ix)] $ \displaystyle \sup_{a<y<x} B_1(y;\alpha, \beta) \approx \sup_{a<y<x} \left({\int_{y}^{x}g(t)F^{\frac{\alpha+s}{\beta}}(t)\, dt}\right)^{\beta} F^{-s}(y); $
\item[(x)] $\displaystyle \sup_{a<y<x} B_1(y;\alpha, \beta) \approx \sup_{a<y<x} B_7(y;\alpha, \beta, s); $
\item[(xi)] $\displaystyle \sup_{a<y<x} B_1(y;\alpha, \beta) \approx \sup_{a<y<x}  \left(\int_{y}^{x}g(t) F ^{\frac{\alpha}{\beta-s}}(t)\, dt\right)^{\beta-s}G^{s}(y) , \hskip 0.5em  \beta>s; $
\item[(xii)] $\displaystyle \sup_{a<y<x} B_1(y;\alpha, \beta) \approx \sup_{a<y<x} B_{11}(y;\alpha, \beta, s), \hskip 0.5em \beta<s; $
\item[(xiii)] $\displaystyle \sup_{a<y<x} B_1(y;\alpha, \beta) \approx \inf_{h \geq 0}\sup_{a<y<x} B_{13}(y;\alpha, \beta, s; h), \hskip 0.5em \alpha<s; $
\item[(xiv)] $\displaystyle \sup_{a<y<x} B_1(y;\alpha, \beta) \approx \inf_{h \geq 0} \sup_{a<y<x}h(y) ^{-s} \left(\int_y^x g(z)\Big(h(z)+F(z)\Big)^{\frac{\alpha+s}{\beta}}dz\right)^{\beta} $.
\end{itemize}
\end{theorem}

\begin{proof}
$(i)$ Let us consider the function
\begin{align*}
f_x(y):= \chi_{(x,b)}(y)f(y).
\end{align*}
By using the equivalence $B_1(\alpha, \beta) \approx B_2(\alpha, \beta, s)$ from Theorem A, we get
\begin{align*}
\sup_{x<y<b} B_1(y;\alpha, \beta)&=\sup_{x<y<b}F^{\alpha}(y)G^{\beta}(y) \\
&= \sup_{a<y<b} \left(\int_{y}^{b}f_x(t)\, dt\right)^{\alpha}\left(\int_{a}^{y}g(t)\, dt\right)^{\beta} \\
&\approx \sup_{a<y<b} \left({\int_{y}^{b}f_x(t)G^{\frac{\beta-s}{\alpha}}(t)\, dt}\right)^{\alpha}G^{s}(y) \\
&=\sup_{x<y<b} \left({\int_{y}^{b}f(t)G^{\frac{\beta-s}{\alpha}}(t)\, dt}\right)^{\alpha}G^{s}(y) \\
&=\sup_{x<y<b} B_2(y;\alpha, \beta, s).
\end{align*}

$(ii)$ Let us consider the function
\begin{align*}
f_x(y):= \chi_{(x,b)}(y)f(y).
\end{align*}
By using the equivalence $B_1(\alpha, \beta) \approx B_4(\alpha, \beta, s)$ from Theorem A, we get
\begin{align*}
\sup_{x<y<b} B_1(y;\alpha, \beta)&=\sup_{x<y<b}F^{\alpha}(y)G^{\beta}(y) \\
&= \sup_{a<y<b} \left(\int_{y}^{b}f_x(t)\, dt\right)^{\alpha}\left(\int_{a}^{y}g(t)\, dt\right)^{\beta} \\
&\approx \sup_{a<y<b} \left(\int_{a}^{y}f_x(t)G^{\frac{\beta+s}{\alpha}}(t)\, dt\right)^{\alpha}G^{-s}(y) \\
&=\sup_{x<y<b}  \left(\int_{x}^{y}f(t)G^{\frac{\beta+s}{\alpha}}(t)\, dt\right)^{\alpha}G^{-s}(y). 
\end{align*}

$(iii)$ Let us consider the function
\begin{align*}
f_x(y):= \chi_{(x,b)}(y)f(y).
\end{align*}
By using the equivalence $B_1(\alpha, \beta) \approx B_6(\alpha, \beta, s)$ from Theorem A, we get
\begin{align*}
\sup_{x<y<b} B_1(y;\alpha, \beta)&=\sup_{x<y<b}F^{\alpha}(y)G^{\beta}(y) \\
&= \sup_{a<y<b} \left(\int_{y}^{b}f_x(t)\, dt\right)^{\alpha}\left(\int_{a}^{y}g(t)\, dt\right)^{\beta} \\
&\approx \sup_{a<y<b} \left({\int_{y}^{b}f_x(t)G^{\frac{\beta}{\alpha+s}}(t)\, dt}\right)^{\alpha+s}F^{-s}(y) \\
&=\sup_{x<y<b} \left({\int_{y}^{b}f(t)G^{\frac{\beta}{\alpha+s}}(t)\, dt}\right)^{\alpha+s} F^{-s}(y) \\
&=\sup_{x<y<b} B_6(y;\alpha, \beta, s).
\end{align*}

$(iv)$ Let us consider the function
\begin{align*}
f_x(y):= \chi_{(x,b)}(y)f(y).
\end{align*}
By using the equivalence $B_1(\alpha, \beta) \approx B_8(\alpha, \beta, s)$ from Theorem A, we get
\begin{align*}
\sup_{x<y<b} B_1(y;\alpha, \beta)&=\sup_{x<y<b}F^{\alpha}(y)G^{\beta}(y) \\
&= \sup_{a<y<b} \left(\int_{y}^{b}f_x(t)\, dt\right)^{\alpha}\left(\int_{a}^{y}g(t)\, dt\right)^{\beta} \\
&\approx \sup_{a<y<b} \left(\int_{a}^{y}f_x(t)G^{\frac{\beta}{\alpha-s}}(t)\, dt\right)^{\alpha-s}F^{s}(y) \\
&=\sup_{x<y<b}  \left(\int_{x}^{y}f(t)G^{\frac{\beta}{\alpha-s}}(t)\, dt\right)^{\alpha-s}F^{s}(y). 
\end{align*}

$(v)$ Let us consider the function
\begin{align*}
f_x(y):= \chi_{(x,b)}(y)f(y).
\end{align*}
By using the equivalence $B_1(\alpha, \beta) \approx B_9(\alpha, \beta, s)$ from Theorem A, we get
\begin{align*}
\sup_{x<y<b} B_1(y;\alpha, \beta)&=\sup_{x<y<b}F^{\alpha}(y)G^{\beta}(y) \\
&= \sup_{a<y<b} \left(\int_{y}^{b}f_x(t)\, dt\right)^{\alpha}\left(\int_{a}^{y}g(t)\, dt\right)^{\beta} \\
&\approx \sup_{a<y<b} \left({\int_{y}^{b}f_x(t)G^{\frac{\beta}{\alpha-s}}(t)\, dt}\right)^{\alpha-s}F^{s}(y) \\
&=\sup_{x<y<b} \left({\int_{y}^{b}f(t)G^{\frac{\beta}{\alpha-s}}(t)\, dt}\right)^{\alpha-s} F^{s}(y) \\
&=\sup_{x<y<b} B_9(y;\alpha, \beta, s).
\end{align*}

$(vi)$ Let us consider the function
\begin{align*}
f_x(y):= \chi_{(x,b)}(y)f(y).
\end{align*}
By using the equivalence $B_1(\alpha, \beta) \approx B_{12}(\alpha, \beta, s)$ from Theorem A, we get
\begin{align*} 
 \sup_{x<y<b} B_1(y;\alpha, \beta)&=\sup_{x<y<b} F^\alpha(y)G^\beta(y) \\
 &=\sup_{a<y<b} \left(\int_{y}^{b}f_x(t)\, dt\right)^{\alpha}\left(\int_{a}^{y}g(t)\, dt\right)^{\beta} \\
&\approx \inf_{h \geq 0}\sup_{a<y<b} \Big(h(y)+G(y)\Big)^{s} \left(\int_y^b f_x(z)h(z)^{\frac{\beta-s}{\alpha}}dz\right)^{\alpha} \\
 &= \inf_{h \geq 0}\sup_{x<y<b} \Big(h(y)+G(y)\Big)^{s}\left(\int_y^b f(z) h(z)^{\frac{\beta-s}{\alpha}}dz\right)^{\alpha} \\
 &=\inf_{h \geq 0}\sup_{x<y<b} B_{12}(y;\alpha, \beta, s; h).
\end{align*}

$(vii)$ Let us consider the function
\begin{align*}
f_x(y):= \chi_{(x,b)}(y)f(y).
\end{align*}
By using the equivalence $B_1(\alpha, \beta) \approx B_{14}(\alpha, \beta, s)$ from Theorem A, we get
\begin{align*}
\sup_{x<y<b} B_1(y;\alpha, \beta)&=\sup_{x<y<b} F^\alpha(y)G^\beta(y) \\
 &=\sup_{a<y<b} \left(\int_{y}^{b}f_x(t)\, dt\right)^{\alpha}\left(\int_{a}^{y}g(t)\, dt\right)^{\beta} \\
 & \approx \inf_{h \geq 0} \sup_{a<y<b}h(y) ^{-s} \left(\int_a^y f_x(z)\Big(h(z)+G(z)\Big)^{\frac{\beta+s}{\alpha}}dz\right)^{\alpha} \\
 &= \inf_{h \geq 0} \sup_{x<y<b}h(y) ^{-s} \left(\int_x^y f(z)\Big(h(z)+G(z)\Big)^{\frac{\beta+s}{\alpha}}dz\right)^{\alpha}.
\end{align*}

$(viii)$ Let us consider the function
\begin{align*}
g_x(y):= \chi_{(a,x)}(y)g(y).
\end{align*}
By using the equivalence $B_1(\alpha, \beta) \approx B_3(\alpha, \beta, s)$ from Theorem A, we get
\begin{align*}
\sup_{a<y<x} B_1(y;\alpha, \beta)&=\sup_{a<y<x}F^{\alpha}(y)G^{\beta}(y) \\
&= \sup_{a<y<b} \left(\int_{y}^{b}f(t)\, dt\right)^{\alpha}\left(\int_{a}^{y}g_x(t)\, dt\right)^{\beta} \\
&\approx \sup_{a<y<b} \left({\int_{a}^{y}g_x(t)F^{\frac{\alpha-s}{\beta}}(t)\, dt}\right)^{\beta} F^{s}(y) \\
&=\sup_{a<y<x} \left({\int_{a}^{y}g(t)F^{\frac{\alpha-s}{\beta}}(t)\, dt}\right)^{\beta} F^{s}(y) \\
&=\sup_{a<y<x} B_3(y;\alpha, \beta, s).
\end{align*}

$(ix)$ Let us consider the function
\begin{align*}
g_x(y):= \chi_{(a,x)}(y)g(y).
\end{align*}
By using the equivalence $B_1(\alpha, \beta) \approx B_5(\alpha, \beta, s)$ from Theorem A, we get
\begin{align*}
\sup_{a<y<x} B_1(y;\alpha, \beta)&=\sup_{a<y<x}F^{\alpha}(y)G^{\beta}(y) \\
&= \sup_{a<y<b} \left(\int_{y}^{b}f(t)\, dt\right)^{\alpha}\left(\int_{a}^{y}g_x(t)\, dt\right)^{\beta} \\
&\approx \sup_{a<y<b} \left({\int_{y}^{b}g_x(t)F^{\frac{\alpha+s}{\beta}}(t) \, dt}\right)^{\beta}F^{-s}(y) \\
&=\sup_{a<y<x} \left({\int_{y}^{x}g(t)F^{\frac{\alpha+s}{\beta}}(t)\, dt}\right)^{\beta} F^{-s}(y).
\end{align*}

$(x)$ Let us consider the function
\begin{align*}
g_x(y):= \chi_{(a,x)}(y)g(y).
\end{align*}
By using the equivalence $B_1(\alpha, \beta) \approx B_7(\alpha, \beta, s)$ from Theorem A, we get
\begin{align*}
\sup_{a<y<x} B_1(y;\alpha, \beta)&=\sup_{a<y<x}F^{\alpha}(y)G^{\beta}(y) \\
&= \sup_{a<y<b} \left(\int_{y}^{b}f(t)\, dt\right)^{\alpha}\left(\int_{a}^{y}g_x(t)\, dt\right)^{\beta} \\
&\approx \sup_{a<y<b} \left({\int_{a}^{y}g_x(t)F^{\frac{\alpha}{\beta+s}}(t)\, dt}\right)^{\beta+s}G^{-s}(y) \\
&=\sup_{a<y<x} \left({\int_{a}^{y}g(t)F^{\frac{\alpha}{\beta+s}}(t)\, dt}\right)^{\beta+s} G^{-s}(y) \\
&=\sup_{a<y<x} B_7(y;\alpha, \beta, s).
\end{align*}

$(xi)$ Let us consider the function
\begin{align*}
g_x(y):= \chi_{(a,x)}(y)g(y).
\end{align*}
By using the equivalence $B_1(\alpha, \beta) \approx B_{10}(\alpha, \beta, s)$ from Theorem A, we get
\begin{align*}
\sup_{a<y<x} B_1(y;\alpha, \beta)&=\sup_{a<y<x}F^{\alpha}(y)G^{\beta}(y) \\
&= \sup_{a<y<b} \left(\int_{y}^{b}f(t)\, dt\right)^{\alpha}\left(\int_{a}^{y}g_x(t)\, dt\right)^{\beta} \\
&\approx \sup_{a<y<b} \left(\int_{y}^{b}g_x(t)F^{\frac{\alpha}{\beta-s}}(t)\, dt\right)^{\beta-s}G^{s}(y) \\
&=\sup_{a<y<x}  \left(\int_{y}^{x}g(t)F^{\frac{\alpha}{\beta-s}}(t)\, dt\right)^{\beta-s}G^{s}(y). 
\end{align*}

$(xii)$ Let us consider the function
\begin{align*}
g_x(y):= \chi_{(a,x)}(y)g(y).
\end{align*}
By using the equivalence $B_1(\alpha, \beta) \approx B_{11}(\alpha, \beta, s)$ from Theorem A, we get
\begin{align*}
\sup_{a<y<x} B_1(y;\alpha, \beta)&=\sup_{a<y<x}F^{\alpha}(y)G^{\beta}(y) \\
&= \sup_{a<y<b} \left(\int_{y}^{b}f(t)\, dt\right)^{\alpha}\left(\int_{a}^{y}g_x(t)\, dt\right)^{\beta} \\
&\approx \sup_{a<y<b} \left({\int_{a}^{y}g_x(t)F^{\frac{\alpha}{\beta-s}}(t)\, dt}\right)^{\beta-s}G^{s}(y) \\
&=\sup_{a<y<x} \left({\int_{a}^{y}g(t)F^{\frac{\alpha}{\beta-s}}(t)\, dt}\right)^{\beta-s} G^{s}(y) \\
&=\sup_{a<y<x} B_{11}(y;\alpha, \beta, s).
\end{align*}

$(xiii)$ Let us consider the function
\begin{align*}
g_x(y):= \chi_{(a,x)}(y)g(y).
\end{align*}
By using the equivalence $B_1(\alpha, \beta) \approx B_{13}(\alpha, \beta, s)$ from Theorem A, we get
\begin{align*} 
 \sup_{a<y<x} B_1(y;\alpha, \beta)&=\sup_{x<y<b} F^\alpha(y)G^\beta(y) \\
 &=\sup_{a<y<b} \left(\int_{y}^{b}f(t)\, dt\right)^{\alpha}\left(\int_{a}^{y}g_x(t)\, dt\right)^{\beta} \\
&\approx \inf_{h \geq 0}\sup_{a<y<b} \Big(h(y)+F(y)\Big)^{s} \left(\int_a^y g_x(z)h(z)^{\frac{\alpha-s}{\beta}}dz\right)^{\beta} \\
 &= \inf_{h \geq 0}\sup_{a<y<x} \Big(h(y)+F(y)\Big)^{s}\left(\int_a^y g(z) h(z)^{\frac{\alpha-s}{\beta}}dz\right)^{\beta} \\
 &=\inf_{h \geq 0}\sup_{a<y<x} B_{13}(y;\alpha, \beta, s; h).
\end{align*}

$(xiv)$ Let us consider the function
\begin{align*}
g_x(y):= \chi_{(a,x)}(y)g(y).
\end{align*}
By using the equivalence $B_1(\alpha, \beta) \approx B_{15}(\alpha, \beta, s)$ from Theorem A, we get
\begin{align*}
\sup_{a<y<x} B_1(y;\alpha, \beta)&=\sup_{a<y<x} F^\alpha(y)G^\beta(y) \\
 &=\sup_{a<y<b} \left(\int_{y}^{b}f(t)\, dt\right)^{\alpha}\left(\int_{a}^{y}g_x(t)\, dt\right)^{\beta} \\
 & \approx \inf_{h \geq 0} \sup_{a<y<b}h(y) ^{-s} \left(\int_y^b g_x(z)\Big(h(z)+F(z)\Big)^{\frac{\alpha+s}{\beta}}dz\right)^{\beta} \\
 &= \inf_{h \geq 0} \sup_{a<y<x}h(y) ^{-s} \left(\int_y^x g(z)\Big(h(z)+F(z)\Big)^{\frac{\alpha+s}{\beta}}dz\right)^{\beta}.
\end{align*}
\end{proof}

\begin{remark}
	Theorem \ref{IThA} suggests that the equivalence is independent of $x$.
\end{remark}

\begin{remark}
Theorem \ref{IThA} demonstrates that, in Theorem A, the supremum over the interval $(a,b)$ can be considered in certain truncated intervals. This consideration will be useful in the results of next section. Moreover, Theorem A can be obtained by Theorem \ref{IThA} by taking $x=a$ or $x=b$ appropriately.
\end{remark}

\section{Equivalence Theorem for Bilinear Hardy Inequality}

We prove the following:

\begin{theorem}\label{t2.1}
Let $-\infty \leq a < b \leq \infty$, $\alpha, \beta, \gamma, s, s_{1}, s_{2}$ be positive numbers and $f, g, h, h_1, h_2$ $\in \mathfrak{M}^{+}$. Let $F,G$ be as in \eqref{e1.4} and denote
\begin{equation} \label{e2.1}
  H(x):= \int_{a}^{x}h(t)\, dt.
\end{equation}
Assume that $F(x)$, $G(x)$ and $H(x)$ are finite for all $x \in (a,b)$. Consider
\allowdisplaybreaks[4]
\begin{align*}
&\tilde{B}_{1}(x; \alpha, \beta, \gamma):= F^{\alpha}(x)G^{\beta}(x)H^{\gamma}(x);\\
&\tilde{B}_{2}(x; \alpha, \beta, \gamma, s_{1},s_{2}):= {\displaystyle \left({\int_{x}^{b}f(t)G^{\frac{\beta-s_{1}}{\alpha}}(t)H^{\frac{\gamma-s_{2}}{\alpha}}(t)\, dt}\right)^{\alpha}G^{s_{1}}(x)H^{s_{2}}(x)};\\
&\tilde{B}_{3}(x; \alpha, \beta, \gamma, s_{1},s_{2}):= {\displaystyle \left({\int_{a}^{x}f(t)G^{\frac{\beta+s_{1}}{\alpha}}(t)H^{\frac{\gamma+s_{2}}{\alpha}}(t)\, dt}\right)^{\alpha}G^{-s_{1}}(x)H^{-s_{2}}(x)} ; \\
&\tilde{B}_4(x; \alpha, \beta, \gamma, s_1,s_2):= \left(\int_{a}^{x}f(t)G^{\frac{\beta}{\alpha+s_1}}(t)H^{\frac{\gamma-s_2}{\alpha+s_1}}(t)\, dt\right)^{\alpha+s_1}F^{-s_1}(x)H^{s_2}(x); \\
&\tilde{B}_5(x; \alpha, \beta, \gamma, s_1,s_2):= \left(\int_{a}^{x}f(t)G^{\frac{\beta-s_2}{\alpha+s_1}}(t)H^{\frac{\gamma}{\alpha+s_1}}(t)\, dt\right)^{\alpha+s_1}F^{-s_1}(x)G^{s_2}(x); \\
&\tilde{B}_6(x; \alpha, \beta, \gamma, s_1,s_2):= \left(\int_{a}^{x}f(t)G^{\frac{\beta}{\alpha-s_1}}(t)H^{\frac{\gamma+s_2}{\alpha-s_1}}(t)\, dt\right)^{\alpha-s_1}F^{s_1}(x)H^{-s_2}(x), \quad \alpha > s_1; \\
&\tilde{B}_7(x; \alpha, \beta, \gamma, s_1,s_2):= \left(\int_{a}^{x}f(t)G^{\frac{\beta+s_2}{\alpha-s_1}}(t)H^{\frac{\gamma}{\alpha-s_1}}(t)\, dt\right)^{\alpha-s_1}F^{s_1}(x)G^{-s_2}(x),  \quad \alpha > s_1; \\
&\tilde{B}_{8}(x; \alpha, \beta, \gamma, s_1,s_2):= \left(\int_{x}^{b}f(t)G^{\frac{\beta}{\alpha-s_1}}(t)\, dt\right)^{\alpha-s_1} H^{s_2}(x)\left(\int_{x}^{b}f(t)H^{\frac{\gamma-s_2}{s_1}}(t)\, dt\right)^{s_1}, \quad \alpha < s_1; \\
&\tilde{B}_{9}(x; \alpha, \beta, \gamma, s_1,s_2):= \left(\int_{x}^{b}f(t)H^{\frac{\gamma}{\alpha-s_1}}(t)\, dt\right)^{\alpha-s_1} G^{s_2}(x)\left(\int_{x}^{b}f(t)G^{\frac{\beta-s_2}{s_1}}(t)\, dt\right)^{s_1}, \quad \alpha < s_1; \\
&\tilde{B}_{10}(x; \alpha, \beta, \gamma, s_1, s_2):= \left(\int_{a}^{x}g(t)F^{\frac{\alpha/2}{\beta+s_1}}(t)\, dt\right)^{\beta+s_1}\left(\int_{a}^{x}h(t)F^{\frac{\alpha/2}{\gamma+s_2}}(t)\, dt\right)^{\gamma+s_2}G^{-s_1}(x)H^{-s_2}(x); \\
&\tilde{B}_{11}(x; \alpha, \beta, \gamma, s):= \left(\int_{x}^{b}f(t)G^{\frac{\beta(1-s)}{\alpha}}(t)H^{\frac{\gamma(1-s)}{\alpha}}(t)\, dt\right)^{\alpha}G^{\beta s}(x)H^{\gamma s}(x); \\
&\tilde{B}_{12}(x; \alpha, \beta, \gamma, s):= \left(\int_{a}^{x}f(t)G^{\frac{\beta(1+s)}{\alpha}}(t)H^{\frac{\gamma(1+s)}{\alpha}}(t)\, dt\right)^{\alpha}G^{-\beta s}(x)H^{-\gamma s}; \\
&\tilde{B}_{13}(x; \alpha, \beta, \gamma, s):= \left(\int_{x}^{b}f(t)G^{\frac{\beta}{\alpha+s}}(t)H^{\frac{\gamma}{\alpha+s}}(t)\, dt\right)^{\alpha+s}F^{-s}(x); \\
&\tilde{B}_{14}(x; \alpha, \beta, \gamma, s):= \left(\int_{a}^{x}f(t)G^{\frac{\beta}{\alpha-s}}(t)H^{\frac{\gamma}{\alpha-s}}(t)\, dt\right)^{\alpha-s}F^{s}(x), \quad \alpha > s; \\
&\tilde{B}_{15}(x; \alpha, \beta, \gamma, s):= \left(\int_{x}^{b}f(t)G^{\frac{\beta}{\alpha-s}}(t)H^{\frac{\gamma}{\alpha-s}}(t)\, dt\right)^{\alpha-s}F^{s}(x), \quad \alpha < s; \\
&\tilde{B}_{16}(x; \alpha, \beta, \gamma, s; h_1):= \left(\int_{x}^{b}f(t)h_1^{\frac{\gamma-s}{\alpha}}(t)\, dt\right)^{\alpha} G^{\beta}(x) \Big(h_1(x)+H(x)\Big)^s, \quad \gamma<s;\\
&\tilde{B}_{17}(x; \alpha, \beta, \gamma, s; h_1):= \left(\int_{x}^{b}f(t)h_1^{\frac{\beta-s}{\alpha}}(t)\, dt\right)^{\alpha}H^{\gamma}(x)\Big(h_1(x)+G(x)\Big)^s, \quad \beta<s; \\
& \tilde{B}_{18}(x; \alpha, \beta, \gamma, s_1, s_2; h_1, h_2):= \left(\int_{x}^{b}f(t)h_1^{\frac{\beta-s_1}{\alpha}}(t)h_2^{\frac{\gamma-s_2}{\alpha}}(t)\, dt\right)^{\alpha} \left\{ \sup_{a< y <x}\Big(h_1(y)+G(y)\Big)^{s_1}\right\} \times \\
& \hskip 15em \times  \Big(h_2(x)+H(x)\Big)^{s_2}, \quad \beta < s_1, \, \gamma < s_2 ; \\
& \tilde{B}_{19}(x; \alpha, \beta, \gamma, s_1, s_2; h_1, h_2):= \left(\int_{x}^{b}f(t)h_1^{\frac{\beta-s_1}{\alpha}}(t)h_2^{\frac{\gamma-s_2}{\alpha}}(t)\, dt\right)^{\alpha}  \Big(h_1(x)+G(x)\Big)^{s_1} \times \\
& \hskip 15em \times \left\{ \sup_{a< y <x} \Big(h_2(y)+H(y)\Big)^{s_2}\right\} , \quad \beta < s_1, \, \gamma < s_2; \\
& \tilde{B}_{20}(x; \alpha, \beta, \gamma, s_1, s_2; h_1, h_2):= \left(\int_{x}^{b}f(t)h_1^{\frac{\beta-s_1}{\alpha}}(t)h_2^{\frac{\gamma-s_2}{\alpha}}(t)\, dt\right)^{\alpha} \left\{ \sup_{a< y <x} \Big(h_1(y)+G(y)\Big)^{s_1}\right\} \times \\
& \hskip 15em \times \left\{ \sup_{a< y <x}\Big(h_2(y)+H(y)\Big)^{s_2}\right\} , \quad \beta < s_1, \, \gamma < s_2 ; \\
& \tilde{B}_{21}(x; \alpha, \beta, \gamma, s_1, s_2; h_1, h_2):=\left(\int_{a}^{x}f(t)\Big(h_1(t)+G(t)\Big)^{\frac{\beta+s_1}{\alpha}}\Big(h_2(t)+H(t)\Big)^{\frac{\gamma+s_2}{\alpha}}\, dt\right)^{\alpha}\, \times \\
& \hskip 25em \times \left\{ \sup_{x<y<b}h_1^{-s_1}(y)\right\} h_2^{-s_2}(x); \\
&  \tilde{B}_{22}(x; \alpha, \beta, \gamma, s_1, s_2; h_1, h_2):= \left(\int_{a}^{x}f(t)\Big(h_1(t)+G(t)\Big)^{\frac{\beta+s_1}{\alpha}}\Big(h_2(t)+H(t)\Big)^{\frac{\gamma+s_2}{\alpha}}\, dt\right)^{\alpha}\, \times \\
& \hskip 25em \times h_1^{-s_1}(x) \left\{ \sup_{x<y<b}h_2^{-s_2}(y)\right\}; \\
& \tilde{B}_{23}(x; \alpha, \beta, \gamma, s_1, s_2; h_1, h_2):= \left(\int_{a}^{x}f(t)\Big(h_1(t)+G(t)\Big)^{\frac{\beta+s_1}{\alpha}}\Big(h_2(t)+H(t)\Big)^{\frac{\gamma+s_2}{\alpha}}\, dt\right)^{\alpha}\, \times \\
& \hskip 22em \times \left\{\sup_{x<y<b}h_1^{-s_1}(y)\right\} \left\{ \sup_{x<y<b}h_2^{-s_2}(y)\right\}.
\end{align*}
The numbers
\begin{align*}
& \tilde{B}_1(\alpha, \beta, \gamma):= \sup_{a < x < b}\tilde{B}_1(x; \alpha, \beta, \gamma),\\
& \tilde{B}_i(\alpha, \beta, \gamma, s_1, s_2):=  \sup_{a < x < b} \tilde{B}_i(x; \alpha, \beta, \gamma, s_1,s_2); \quad (i=2, 3, 4, 5, 6, 7, 8, 9, 10), \\
& \tilde{B}_i(\alpha, \beta, \gamma, s):= \sup_{a < x < b} \tilde{B}_i(x; \alpha, \beta, \gamma, s); \quad (i=11, 12, 13, 14, 15), \\
& \tilde{B}_i(\alpha, \beta, \gamma, s):= \inf_{h_1 \geq 0} \sup_{a < x < b} \tilde{B}_i(x; \alpha, \beta, \gamma, s; h_1); \quad (i=16, 17)
\end{align*}
and
\begin{align*}
\tilde{B}_i(\alpha, \beta, \gamma, s_1, s_2):= \inf_{h_1, h_2 \geq 0} \sup_{a < x < b} \tilde{B}_i(x; \alpha, \beta, \gamma, s_1, s_2; h_1, h_2); \quad (i=18, 19, 20, 21, 22, 23)
\end{align*}
are mutually equivalent in the sense that if one number is finite then the others are so.
\end{theorem}

\begin{proof}
The equivalence $\tilde{B}_1(\alpha, \beta, \gamma) \approx \tilde{B}_i(\alpha, \beta, \gamma, s_1,s_2) \approx \tilde{B}_j(\alpha, \beta, \gamma, s)$, $(i=2, 3, 10)$, $(j=13, 14, 15)$ has been proved in \cite{KPS}. Here, we give different proofs for the equivalence
\begin{equation*}
\tilde{B}_1(\alpha, \beta, \gamma) \approx \tilde{B}_k(\alpha, \beta, \gamma, s_1,s_2), \, (k=2, 3)
\end{equation*}
and prove the remaining ones.
\\~\\
$(i)$ $\tilde{B}_1(\alpha, \beta, \gamma) \approx \tilde{B}_2(\alpha, \beta, \gamma, s_1, s_2)$
\\~\\
By using the equivalence $B_{1}(\alpha, \beta) \approx B_{2}(\alpha, \beta, s)$ from Theorem A and in view of Theorem \ref{IThA} (i)
\begin{align*}
\sup_{x<y<b}F^{\alpha}(y)G^{\beta}(y) \approx \sup_{x<y<b} \left({\int_{y}^{b}f(t)G^{\frac{\beta-s_1}{\alpha}}(t)\, dt}\right)^{\alpha}G^{s_1}(y),
\end{align*}
we get
\allowdisplaybreaks[4]
\begin{align*}
\tilde{B}_1(\alpha, \beta, \gamma)=&\sup_{a<x<b}\tilde{B}_1(x; \alpha, \beta, \gamma)\\
=&\sup_{a<x<b}F^{\alpha}(x)G^{\beta}(x)H^{\gamma}(x) \\
=&\sup_{a<x<b}H^{\gamma}(x)\sup_{x<y<b}F^{\alpha}(y)G^{\beta}(y)\\
\approx &\sup_{a<x<b}H^{\gamma}(x)\sup_{x<y<b}G^{s_1}(y)\left(\int_{y}^{b}f(t)G^{\frac{\beta-s_1}{\alpha}}(t)\, dt\right)^{\alpha} \\
=&\sup_{a<x<b}G^{s_1}(x)\sup_{x<y<b}H^{\gamma}(y)\left(\int_{y}^{b}f(t)G^{\frac{\beta-s_1}{\alpha}}(t)\, dt\right)^{\alpha} \\
\approx &\sup_{a<x<b}G^{s_1}(y)\sup_{x<y<b}H^{s_2}(y)\left(\int_{y}^{b}f(t)G^{\frac{\beta-s_1}{\alpha}}(t)H^{\frac{\gamma-s_2}{\alpha}}(t)\, dt\right)^{\alpha} \\
=&\sup_{a<x<b}\left(\int_{x}^{b}f(t)G^{\frac{\beta-s_1}{\alpha}}(t)H^{\frac{\gamma-s_2}{\alpha}}(t)\, dt\right)^{\alpha}G^{s_1}(x)H^{s_2}(x) \\
=&\sup_{a<x<b}\tilde{B}_2(x; \alpha, \beta, \gamma, s_{1},s_{2}) \\
=&\tilde{B}_2(\alpha, \beta, \gamma, s_1, s_2).
\end{align*}
\\~\\
$(ii)$ $\tilde{B}_1(\alpha, \beta, \gamma) \approx \tilde{B}_3(\alpha, \beta, \gamma, s_1, s_2)$
\\~\\
By using the equivalence $B_{1}(\alpha, \beta) \approx B_{4}(\alpha, \beta, s)$ from Theorem A and in view of Theorem \ref{IThA} (ii)
\begin{align*}
\sup_{x<y<b} F^{\alpha}(y)G^{\beta}(y)  \approx \sup_{x<y<b}  \left(\int_{x}^{y}f(t)G^{\frac{\beta+s_1}{\alpha}}(t)\, dt\right)^{\alpha}G^{-s_1}(y),
\end{align*}
we get
\begin{align*}
\tilde{B}_1(\alpha, \beta, \gamma)=&\sup_{a<x<b}\tilde{B}_1(x; \alpha, \beta, \gamma)\\
=&\sup_{a<x<b}F^{\alpha}(x)G^{\beta}(x)H^{\gamma}(x) \\
=&\sup_{a<x<b}H^{\gamma}(x)\sup_{x<y<b}F^{\alpha}(y)G^{\beta}(y)\\
\approx &\sup_{a<x<b}H^{\gamma}(x)\sup_{x<y<b}G^{-s_1}(y)\left(\int_{x}^{y}f(t)G^{\frac{\beta+s_1}{\alpha}}(t)\, dt\right)^{\alpha} \\
=&\sup_{a<y<b}G^{-s_1}(y)\sup_{a<x<y}H^{\gamma}(x)\left(\int_{x}^{y}f(t)G^{\frac{\beta+s_1}{\alpha}}(t)\, dt\right)^{\alpha} \\
\approx &\sup_{a<y<b}G^{-s_1}(y)\sup_{a<x<y}H^{-s_2}(x)\left(\int_{a}^{x}f(t)G^{\frac{\beta+s_1}{\alpha}}(t)H^{\frac{\gamma+s_2}{\alpha}}(t)\, dt\right)^{\alpha} \\
=&\sup_{a<y<b}\left(\int_{a}^{y}f(t)G^{\frac{\beta+s_1}{\alpha}}(t)H^{\frac{\gamma+s_2}{\alpha}}(t)\, dt\right)^{\alpha}G^{-s_1}(y)H^{-s_2}(y) \\
=&\sup_{a<x<b}\left(\int_{a}^{x}f(t)G^{\frac{\beta+s_1}{\alpha}}(t)H^{\frac{\gamma+s_2}{\alpha}}(t)\, dt\right)^{\alpha}G^{-s_1}(x)H^{-s_2}(x) \\
=&\sup_{a<x<b}\tilde{B}_3(x; \alpha, \beta, \gamma, s_{1},s_{2}) \\
=&\tilde{B}_3(\alpha, \beta, \gamma, s_1, s_2).
\end{align*}
\\~\\
$(iii)$ $\tilde{B}_1(\alpha, \beta, \gamma) \approx \tilde{B}_4(\alpha, \beta, \gamma, s_1, s_2)$
\\~\\
By using the equivalences $B_{1}(\alpha, \beta) \approx B_{6}(\alpha, \beta, s)$, $B_{1}(\alpha, \beta) \approx B_{2}(\alpha, \beta, s)$ from Theorem A and in view of Theorem \ref{IThA} (iii)
\begin{align*}
\sup_{x<y<b}F^{\alpha}(y)G^{\beta}(y)\approx  \sup_{x<y<b} \left(\int_{y}^{b}f(t)G^{\frac{\beta}{\alpha+s_1}}(t)\, dt\right)^{\alpha+s_1}F^{-s_1}(y),
\end{align*}
we get
\begin{align*}
\tilde{B}_1(\alpha, \beta, \gamma)=&\sup_{a<x<b}\tilde{B}_1(x; \alpha, \beta, \gamma)\\
=&\sup_{a<x<b}F^{\alpha}(x)G^{\beta}(x)H^{\gamma}(x) \\
= &\sup_{a<x<b}H^{\gamma}(x)\sup_{x<y<b}F^{\alpha}(y)G^{\beta}(y) \\
\approx & \sup_{a<x<b}H^{\gamma}(x) \sup_{x<y<b} \left(\int_{y}^{b}f(t)G^{\frac{\beta}{\alpha+s_1}}(t)\, dt\right)^{\alpha+s_1}F^{-s_1}(y) \\
=&\sup_{a<x<b}F^{-s_1}(x) \sup_{x<y<b} \left(\int_{y}^{b}f(t)G^{\frac{\beta}{\alpha+s_1}}(t)\, dt\right)^{\alpha+s_1}H^{\gamma}(y) \\
\approx & \sup_{a<x<b} F^{-s_1}(x) \sup_{x<y<b} \left(\int_{y}^{b}f(t)G^{\frac{\beta}{\alpha+s_1}}(t)H^{\frac{\gamma-s_2}{\alpha+s_1}}(t)\, dt\right)^{\alpha+s_1}H^{s_2}(y) \\
=& \sup_{a<x<b} F^{-s_1}(x) H^{s_2}(x) \left(\int_{x}^{b}f(t)G^{\frac{\beta}{\alpha+s_1}}(t)H^{\frac{\gamma-s_2}{\alpha+s_1}}(t)\, dt\right)^{\alpha+s_1} \\
=&\sup_{a<x<b}\tilde{B}_4(x; \alpha, \beta, \gamma, s_{1},s_{2}) \\
=&\tilde{B}_4(\alpha, \beta, \gamma, s_1, s_2).
\end{align*}
\\~\\
$(iv)$ $\tilde{B}_1(\alpha, \beta, \gamma) \approx \tilde{B}_5(\alpha, \beta, \gamma, s_1, s_2)$
\\~\\
By using the equivalences $B_{1}(\alpha, \beta) \approx B_{6}(\alpha, \beta, s)$, $B_{1}(\alpha, \beta) \approx B_{2}(\alpha, \beta, s)$ from Theorem A and in view of Theorem \ref{IThA} (iii)
\begin{align*}
\sup_{x<y<b}F^{\alpha}(y)H^{\gamma}(y)\approx  \sup_{x<y<b} \left(\int_{y}^{b}f(t)H^{\frac{\gamma}{\alpha+s_1}}(t)\, dt\right)^{\alpha+s_1}F^{-s_1}(y),
\end{align*}
we get
\begin{align*}
\tilde{B}_1(\alpha, \beta, \gamma)=&\sup_{a<x<b}\tilde{B}_1(x; \alpha, \beta, \gamma)\\
=&\sup_{a<x<b}F^{\alpha}(x)G^{\beta}(x)H^{\gamma}(x) \\
= &\sup_{a<x<b}G^{\beta}(x)\sup_{x<y<b}F^{\alpha}(y)H^{\gamma}(y) \\
\approx & \sup_{a<x<b}G^{\beta}(x) \sup_{x<y<b} \left(\int_{y}^{b}f(t)H^{\frac{\gamma}{\alpha+s_1}}(t)\, dt\right)^{\alpha+s_1}F^{-s_1}(y) \\
=& \sup_{a<x<b}F^{-s_1}(x) \sup_{x<y<b} \left(\int_{y}^{b}f(t)H^{\frac{\gamma}{\alpha+s_1}}(t)\, dt\right)^{\alpha+s_1}G^{\beta}(y) \\
\approx & \sup_{a<x<b} F^{-s_1}(x) \sup_{x<y<b} \left(\int_{y}^{b}f(t)H^{\frac{\gamma}{\alpha+s_1}}(t)G^{\frac{\beta-s_2}{\alpha+s_1}}(t)\, dt\right)^{\alpha+s_1}G^{s_2}(y) \\
=& \sup_{a<x<b} F^{-s_1}(x) G^{s_2}(x) \left(\int_{x}^{b}f(t)H^{\frac{\gamma}{\alpha+s_1}}(t)G^{\frac{\beta-s_2}{\alpha+s_1}}(t)\, dt\right)^{\alpha+s_1} \\
=&\sup_{a<x<b}\tilde{B}_5(x; \alpha, \beta, \gamma, s_{1},s_{2}) \\
=&\tilde{B}_5(\alpha, \beta, \gamma, s_1, s_2).
\end{align*}
\\~\\
$(v)$ $\tilde{B}_1(\alpha, \beta, \gamma) \approx \tilde{B}_6(\alpha, \beta, \gamma, s_1, s_2)$
\\~\\
By using the equivalences $B_{1}(\alpha, \beta) \approx B_{8}(\alpha, \beta, s)$, $B_{1}(\alpha, \beta) \approx B_{4}(\alpha, \beta, s)$  from Theorem A and in view of Theorem \ref{IThA} (iv)
\begin{align*}
\sup_{x<y<b}G^{\beta}(y)F^{\alpha}(y)\approx \sup_{x<y<b}  \left(\int_{x}^{y}f(t)G^{\frac{\beta}{\alpha-s_1}}(t)\, dt\right)^{\alpha-s_1}F^{s_1}(y); \, \, \,  \alpha > s_1,
\end{align*}
we have   
\begin{align*}
\tilde{B}_1(\alpha, \beta, \gamma)=&\sup_{a<x<b}\tilde{B}_1(x; \alpha, \beta, \gamma)\\
=&\sup_{a<x<b}F^{\alpha}(x)G^{\beta}(x)H^{\gamma}(x) \\
=&\sup_{a<x<b}H^{\gamma}(x) \sup_{x<y<b}F^{\alpha}(y)G^{\beta}(y) \\
\approx & \sup_{a<x<b} H^{\gamma}(x) \sup_{x<y<b}\left(\int_{x}^{y}f(t)G^{\frac{\beta}{\alpha-s_1}}(t)\, dt\right)^{\alpha-s_1}F^{s_1}(y) \\
=& \sup_{a<x<b} F^{s_1}(x)\sup_{a<y<x}H^{\gamma}(y)\left(\int_{y}^{x}f(t)G^{\frac{\beta}{\alpha-s_1}}(t)\, dt\right)^{\alpha-s_1}\\
\approx& \sup_{a<x<b} F^{s_1}(x)\sup_{a<y<x}H^{-s_2}(y)\left(\int_{a}^{y}f(t)H^{\frac{\gamma+s_2}{\alpha-s_1}}(t) G^{\frac{\beta}{\alpha-s_1}}(t)\, dt\right)^{\alpha-s_1}\\
=& \sup_{a<x<b} F^{s_1}(x)H^{-s_2}(x)\left(\int_{a}^{x}f(t)H^{\frac{\gamma+s_2}{\alpha-s_1}}(t) G^{\frac{\beta}{\alpha-s_1}}(t)\, dt\right)^{\alpha-s_1}\\
=&\sup_{a<x<b}  \tilde{B}_6(x; \alpha, \beta, \gamma, s_1, s_2) \\
=&\tilde{B}_6(\alpha, \beta, \gamma, s_1, s_2).
\end{align*}
\\~\\
$(vi)$ $\tilde{B}_1(\alpha, \beta, \gamma) \approx \tilde{B}_7(\alpha, \beta, \gamma, s_1, s_2)$
\\~\\
By using the equivalences $B_{1}(\alpha, \beta) \approx B_{8}(\alpha, \beta, s)$, $B_{1}(\alpha, \beta) \approx B_{4}(\alpha, \beta, s)$ from Theorem A and in view of Theorem \ref{IThA} (iv)
\begin{align*}
 \sup_{x<y<b}H^{\gamma}(y)F^{\alpha}(y)\approx \sup_{x<y<b}  \left(\int_{x}^{y}f(t)H^{\frac{\gamma}{\alpha-s_1}}(t)\, dt\right)^{\alpha-s_1}F^{s_1}(y); \, \, \,  \alpha > s_1,
\end{align*}
we have   
\begin{align*}
\tilde{B}_1(\alpha, \beta, \gamma)=&\sup_{a<x<b}\tilde{B}_1(x; \alpha, \beta, \gamma)\\
=&\sup_{a<x<b}F^{\alpha}(x)G^{\beta}(x)H^{\gamma}(x) \\
=&\sup_{a<x<b}G^{\beta}(x) \sup_{x<y<b}F^{\alpha}(y)H^{\gamma}(y) \\
\approx & \sup_{a<x<b} G^{\beta}(x) \sup_{x<y<b}\left(\int_{x}^{y}f(t)H^{\frac{\gamma}{\alpha-s_1}}(t)\, dt\right)^{\alpha-s_1}F^{s_1}(y) \\
=& \sup_{a<x<b} F^{s_1}(x)\sup_{a<y<x}G^{\beta}(y)\left(\int_{y}^{x}f(t)H^{\frac{\gamma}{\alpha-s_1}}(t)\, dt\right)^{\alpha-s_1}\\
\approx& \sup_{a<x<b} F^{s_1}(x)\sup_{a<y<x}G^{-s_2}(y)\left(\int_{a}^{y}f(t)G^{\frac{\beta+s_2}{\alpha-s_1}}(t) H^{\frac{\gamma}{\alpha-s_1}}(t)\, dt\right)^{\alpha-s_1}\\
=& \sup_{a<x<b} F^{s_1}(x)G^{-s_2}(x)\left(\int_{a}^{x}f(t)G^{\frac{\beta+s_2}{\alpha-s_1}}(t) H^{\frac{\gamma}{\alpha-s_1}}(t)\, dt\right)^{\alpha-s_1}\\
=&\sup_{a<x<b}  \tilde{B}_7(x; \alpha, \beta, \gamma, s_1, s_2) \\
=&\tilde{B}_7(\alpha, \beta, \gamma, s_1, s_2).
\end{align*}
\\~\\
$(vii)$ $\tilde{B}_1(\alpha, \beta, \gamma) \approx \tilde{B}_{8}(\alpha, \beta, \gamma, s_1, s_2)$
\\~\\
By using the equivalences $B_{1}(\alpha, \beta) \approx B_{9}(\alpha, \beta, s)$, $B_{1}(\alpha, \beta) \approx B_{2}(\alpha, \beta, s)$ from Theorem A and in view of Theorem \ref{IThA} (v)
\begin{align*}
\sup_{x<y<b}G^{\beta}(y)F^{\alpha}(y)\approx \sup_{x<y<b}  \left(\int_{y}^{b}f(t)G^{\frac{\beta}{\alpha-s_1}}(t)\, dt\right)^{\alpha-s_1}F^{s_1}(y);\, \, \,  \alpha < s_1, 
\end{align*}
we have   
\begin{align*}
\tilde{B}_1(\alpha, \beta, \gamma)=&\sup_{a<x<b}\tilde{B}_1(x; \alpha, \beta, \gamma)\\
=&\sup_{a<x<b}F^{\alpha}(x)G^{\beta}(x)H^{\gamma}(x) \\
=&\sup_{a<x<b}H^{\gamma}(x) \sup_{x<y<b}F^{\alpha}(y)G^{\beta}(y) \\
\approx & \sup_{a<x<b} H^{\gamma}(x) \sup_{x<y<b}\left(\int_{y}^{b}f(t)G^{\frac{\beta}{\alpha-s_1}}(t)\, dt\right)^{\alpha-s_1}F^{s_1}(y) \\
=& \sup_{a<x<b} \left(\int_{x}^{b}f(t)G^{\frac{\beta}{\alpha-s_1}}(t)\, dt\right)^{\alpha-s_1}F^{s_1}(x)H^{\gamma}(x)\\
=&\sup_{a<x<b}\left(\int_{x}^{b}f(t)G^{\frac{\beta}{\alpha-s_1}}(t)\, dt\right)^{\alpha-s_1}
\sup_{x<y<b} F^{s_1}(y)H^{\gamma}(y)  \\
\approx& \sup_{a<x<b}\left(\int_{x}^{b}f(t)G^{\frac{\beta}{\alpha-s_1}}(t)\, dt\right)^{\alpha-s_1} 
\sup_{x<y<b}H^{s_2}(y)\left(\int_{y}^{b}f(t)H^{\frac{\gamma-s_2}{s_1}}(t)\, dt\right)^{s_1}\\
=&\sup_{a<x<b}\left(\int_{x}^{b}f(t)G^{\frac{\beta}{\alpha-s_1}}(t)\, dt\right)^{\alpha-s_1} 
H^{s_2}(x)\left(\int_{x}^{b}f(t)H^{\frac{\gamma-s_2}{s_1}}(t)\, dt\right)^{s_1}\\
=&\sup_{a<x<b}  \tilde{B}_8(x; \alpha, \beta, \gamma, s_1, s_2) \\
=&\tilde{B}_8(\alpha, \beta, \gamma, s_1, s_2).
\end{align*}
\\~\\
$(viii)$ $\tilde{B}_1(\alpha, \beta, \gamma) \approx \tilde{B}_9(\alpha, \beta, \gamma, s_1, s_2)$
\\~\\
By using the equivalences $B_{1}(\alpha, \beta) \approx B_{9}(\alpha, \beta, s)$, $B_{1}(\alpha, \beta) \approx B_{2}(\alpha, \beta, s)$ from Theorem A and in view of Theorem \ref{IThA} (v)
\begin{align*}
\sup_{x<y<b}H^{\gamma}(y)F^{\alpha}(y)\approx \sup_{x<y<b}  \left(\int_{y}^{b}f(t)H^{\frac{\gamma}{\alpha-s_1}}(t)\, dt\right)^{\alpha-s_1}F^{s_1}(y); \, \, \,  \alpha < s_1,
\end{align*}
we have   
\begin{align*}
\tilde{B}_1(\alpha, \beta, \gamma)=&\sup_{a<x<b}\tilde{B}_1(x; \alpha, \beta, \gamma)\\
=&\sup_{a<x<b}F^{\alpha}(x)G^{\beta}(x)H^{\gamma}(x) \\
=&\sup_{a<x<b}G^{\beta}(x) \sup_{x<y<b}F^{\alpha}(y)H^{\gamma}(y) \\
\approx & \sup_{a<x<b} G^{\beta}(x) \sup_{x<y<b}\left(\int_{y}^{b}f(t)H^{\frac{\gamma}{\alpha-s_1}}(t)\, dt\right)^{\alpha-s_1}F^{s_1}(y) \\
=& \sup_{a<x<b} G^{\beta}(x)F^{s_1}(x)\left(\int_{x}^{b}f(t)H^{\frac{\gamma}{\alpha-s_1}}(t)\, dt\right)^{\alpha-s_1}\\
=& \sup_{a<x<b} \left(\int_{x}^{b}f(t)H^{\frac{\gamma}{\alpha-s_1}}(t)\, dt\right)^{\alpha-s_1}\sup_{x<y<b}F^{s_1}(y)
G^{\beta}(y)\\
\approx& \sup_{a<x<b}\left(\int_{x}^{b}f(t)H^{\frac{\gamma}{\alpha-s_1}}(t)\, dt\right)^{\alpha-s_1}\sup_{x<y<b}G^{s_2}(y)\left(\int_{y}^{b}f(t)G^{\frac{\beta-s_2}{s_1}}(t)\, dt\right)^{s_1}
\\
=&  \sup_{a<x<b}\left(\int_{x}^{b}f(t)H^{\frac{\gamma}{\alpha-s_1}}(t)\, dt\right)^{\alpha-s_1}G^{s_2}(x)\left(\int_{x}^{b}f(t)G^{\frac{\beta-s_2}{s_1}}(t)\, dt\right)^{s_1}
\\
=&\sup_{a<x<b}  \tilde{B}_9(x; \alpha, \beta, \gamma, s_1, s_2) \\
=&\tilde{B}_9(\alpha, \beta, \gamma, s_1, s_2).
\end{align*}
\\~\\
$(ix)$    $\tilde{B}_1(\alpha, \beta, \gamma) \approx \tilde{B}_{11}(\alpha, \beta, \gamma, s)$
\\~\\
Easily follows from $B_1(\alpha, \beta) \approx B_{2}(\alpha, \beta, s)$.
\\~\\
$(x)$  $\tilde{B}_1(\alpha, \beta, \gamma) \approx \tilde{B}_{12}(\alpha, \beta, \gamma, s)$
\\~\\
Similarly from $B_1(\alpha, \beta) \approx B_{4}(\alpha, \beta, s)$.
\\~\\
$(xi)$ $\tilde{B}_1(\alpha, \beta, \gamma) \approx \tilde{B}_{16}(\alpha, \beta, \gamma, s)$
\\~\\
By using the equivalences $B_{1}(\alpha, \beta) \approx B_{12}(\alpha, \beta, s)$ from Theorem A and in view of Theorem \ref{IThA} (vi)
\begin{align*}
\sup_{x<y<b}H^{\gamma}(y)F^{\alpha}(y)\approx \inf_{h_1 \geq 0}\sup_{x<y<b}\Big(h_1(y)+H(y)\Big)^{s} \left(\int_{y}^{b}f(t)h_{1}^{\frac{\gamma-s}{\alpha}}(t)\, dt\right)^{\alpha}; \quad \gamma<s,
\end{align*}
we have   
\begin{align*}
\tilde{B}_1(\alpha, \beta, \gamma)=&\sup_{a<x<b}\tilde{B}_1(x; \alpha, \beta, \gamma)\\
=&\sup_{a<x<b}F^{\alpha}(x)G^{\beta}(x)H^{\gamma}(x) \\
=&\sup_{a<x<b}G^{\beta}(x) \sup_{x<y<b}F^{\alpha}(y)H^{\gamma}(y) \\
\approx & \inf_{h_1 \geq 0}\sup_{a<x<b} G^{\beta}(x) \sup_{x<y<b}\Big(h_1(y)+H(y)\Big)^s\left(\int_{y}^{b}f(t)h_1^{\frac{\gamma-s}{\alpha}}(t)\, dt\right)^{\alpha}\\
=& \inf_{h_1 \geq 0}\sup_{a<x<b} \left(\int_{x}^{b}f(t)h_1^{\frac{\gamma-s}{\alpha}}(t)\, dt\right)^{\alpha} G^{\beta}(x)\Big(h_1(x)+H(x)\Big)^{s} \\
=&\inf_{h_1 \geq 0} \sup_{a<x<b}  \tilde{B}_{16}(x; \alpha, \beta, \gamma, s; h_1) \\
=&\tilde{B}_{16}(\alpha, \beta, \gamma, s).
\end{align*}
\\~\\
$(xii)$ $\tilde{B}_1(\alpha, \beta, \gamma) \approx \tilde{B}_{17}(\alpha, \beta, \beta, s)$
\\~\\
By using the equivalences $B_{1}(\alpha, \beta) \approx B_{12}(\alpha, \beta, s)$ from Theorem A and in view of Theorem \ref{IThA} (vi)
\begin{align*}
\sup_{x<y<b}G^{\beta}(y)F^{\alpha}(y)\approx \inf_{h_1 \geq 0}\sup_{x<y<b}\Big(h_1(y)+G(y)\Big)^{s} \left(\int_{y}^{b}f(t)h_{1}^{\frac{\beta-s}{\alpha}}(t)\, dt\right)^{\alpha}; \quad \beta<s,
\end{align*}
we have   
\begin{align*}
\tilde{B}_1(\alpha, \beta, \gamma)=&\sup_{a<x<b}\tilde{B}_1(x; \alpha, \beta, \gamma)\\
=&\sup_{a<x<b}F^{\alpha}(x)G^{\beta}(x)H^{\gamma}(x) \\
=&\sup_{a<x<b}H^{\gamma}(x) \sup_{x<y<b}F^{\alpha}(y)G^{\beta}(y) \\
\approx & \inf_{h_1 \geq 0}\sup_{a<x<b} H^{\gamma}(x) \sup_{x<y<b}\Big(h_1(y)+G(y)\Big)^s\left(\int_{y}^{b}f(t)h_1^{\frac{\beta-s}{\alpha}}(t)\, dt\right)^{\alpha}\\
=& \inf_{h_1 \geq 0}\sup_{a<x<b} \left(\int_{x}^{b}f(t)h_1^{\frac{\beta-s}{\alpha}}(t)\, dt\right)^{\alpha} H^{\gamma}(x)\Big(h_1(x)+G(x)\Big)^{s} \\
=&\inf_{h_1 \geq 0} \sup_{a<x<b}  \tilde{B}_{17}(x; \alpha, \beta, \gamma, s; h_1) \\
=&\tilde{B}_{17}(\alpha, \beta, \gamma, s).
\end{align*}
\\~\\
$(xiii)$ $\tilde{B}_1(\alpha, \beta, \gamma) \approx \tilde{B}_{18}(\alpha, \beta, \gamma, s_1, s_2)$
\\~\\
By using the equivalence $B_{1}(\alpha, \beta) \approx B_{12}(\alpha, \beta, s)$ from Theorem A  and in view of Theorem \ref{IThA} (vi)
\begin{align*}
\sup_{x<y<b} F^\alpha(y)G^\beta(y)
\approx \inf_{h_1>0} \sup_{x<y<b}\Big(h_1(y) +G(y)\Big)^{s} \left(\int_y^b f(z)h_1^{\frac{\beta-s}{\alpha}}(z)dz\right)^{\alpha}; \quad \beta <s
\end{align*}
and also the fact that $H^{\gamma}(x) \geq 0$ as well as $\Big(h_1(x)+G(x)\Big)^{s_1} \geq 0$, we get
\begin{align*}
\tilde{B}_1(\alpha, \beta, \gamma)=&\sup_{a<x<b}\tilde{B}_1(x; \alpha, \beta, \gamma)\\
=&\sup_{a<x<b}F^{\alpha}(x)G^{\beta}(x)H^{\gamma}(x) \\
=&\sup_{a<x<b}H^{\gamma}(x)\sup_{x<y<b}F^{\alpha}(y)G^{\beta}(y) \\
\approx &\sup_{a<x<b}H^{\gamma}(x)\inf_{h_1 \geq 0}\sup_{x<y<b}\left(\int_{y}^{b}f(t)h_1^{\frac{\beta-s_1}{\alpha}}(t)\, dt\right)^{\alpha}\Big(h_1(y)+G(y)\Big)^{s_1} \\
=&\inf_{h_1 \geq 0}\sup_{a<x<b}H^{\gamma}(x)\sup_{x<y<b}\left(\int_{y}^{b}f(t)h_1^{\frac{\beta-s_1}{\alpha}}(t)\, dt\right)^{\alpha}\Big(h_1(y)+G(y)\Big)^{s_1}  \\
\leq &\inf_{h_1 \geq 0}\sup_{a<x<b} \left\{ \sup_{a< y <x}\Big(h_1(y)+G(y)\Big)^{s_1}\right\} \left\{\sup_{x<y<b}\left(\int_{y}^{b}f(t)h_1^{\frac{\beta-s_1}{\alpha}}(t)\, dt\right)^{\alpha}H^{\gamma}(y) \right\}  \\
\approx &\inf_{h_1 \geq 0}\sup_{a<x<b} \left\{ \sup_{a< y <x}\Big(h_1(y)+G(y)\Big)^{s_1}\right\} \, \times  \\
& \hskip 5em  \times \Bigg\{\inf_{h_2 \geq 0}\sup_{x<y<b}\left(\int_{y}^{b}f(t)h_1^{\frac{\beta-s_1}{\alpha}}(t)h_2^{\frac{\gamma-s_2}{\alpha}}(t)\, dt\right)^{\alpha} \Big(h_2(y)+H(y)\Big)^{s_2} \Bigg\}  \\
=&\inf_{h_1 \geq 0}\inf_{h_2 \geq 0}\sup_{a<x<b} \left\{ \sup_{a< y <x}\Big(h_1(y)+G(y)\Big)^{s_1}\right\} \, \times \\
& \hskip 5em  \times \left\{ \sup_{x<y<b}\left(\int_{y}^{b}f(t)h_1^{\frac{\beta-s_1}{\alpha}}(t)h_2^{\frac{\gamma-s_2}{\alpha}}(t)\, dt\right)^{\alpha}\Big(h_2(y)+H(y)\Big)^{s_2} \right\} \\
=&\inf_{h_1, h_2 \geq 0}\sup_{a<x<b}\left(\int_{x}^{b}f(t)h_1^{\frac{\beta-s_1}{\alpha}}(t)h_2^{\frac{\gamma-s_2}{\alpha}}(t)\, dt\right)^{\alpha}\, \times \\
& \hskip 8em \times \left\{ \sup_{a< y <x}\Big(h_1(y)+G(y)\Big)^{s_1}\right\} \Big(h_2(x)+H(x)\Big)^{s_2} \\
=& \inf_{h_1, h_2 \geq 0}\sup_{a<x<b} \tilde{B}_{18}(x; \alpha, \beta, \gamma, s_1, s_2; h_1, h_2) \\
=&\tilde{B}_{18}(\alpha, \beta, \gamma, s_1, s_2).
\end{align*}
Thus
\begin{align} \label{e2.2}
\tilde{B}_1(\alpha, \beta, \gamma) \lesssim \tilde{B}_{18}(\alpha, \beta, \gamma, s_1, s_2).
\end{align}
Since for $h_1(x)= G(x)$ and $h_2(x)=H(x)$, we have
\begin{align*}
\tilde{B}_{18}(x; \alpha, \beta, \gamma, s_1, s_2; G, H)=&2^{s_1+s_2}\left(\int_{x}^{b}f(t)G^{\frac{\beta-s_1}{\alpha}}(t)H^{\frac{\gamma-s_2}{\alpha}}(t)\, dt\right)^{\alpha} G^{s_1}(x) H^{s_2}(x) \\
=& 2^{s_1+s_2} \tilde{B}_{2}(x; \alpha, \beta, \gamma, s_1, s_2) \\
\leq & 2^{s_1+s_2} \sup_{a<x<b}\tilde{B}_{2}(x; \alpha, \beta, \gamma, s_1, s_2) \\
= & 2^{s_1+s_2} \tilde{B}_{2}(\alpha, \beta, \gamma, s_1, s_2) \\
\approx & 2^{s_1+s_2} \tilde{B}_{1}(\alpha, \beta, \gamma).
\end{align*}
The last relation holds by using $\tilde{B}_{1}(\alpha, \beta, \gamma) \approx \tilde{B}_{2}(\alpha, \beta, \gamma, s_1, s_2)$.
Clearly we get that
\begin{align} \label{e2.3}
\tilde{B}_{18}(\alpha, \beta, \gamma, s_1, s_2) \lesssim \tilde{B}_{1}(\alpha, \beta, \gamma).
\end{align}
In view of \eqref{e2.2} and \eqref{e2.3}, we obtain
\begin{align*}
\tilde{B}_1(\alpha, \beta, \gamma) \approx \tilde{B}_{18}(\alpha, \beta, \gamma, s_1, s_2).
\end{align*}
\\~\\
$(xiv)$ $\tilde{B}_1(\alpha, \beta, \gamma) \approx \tilde{B}_{19}(\alpha, \beta, \gamma, s_1, s_2)$
\\~\\
Using the similar reasoning as in $(xiii)$.
\\~\\
$(xv)$ $\tilde{B}_1(\alpha, \beta, \gamma) \approx \tilde{B}_{20}(\alpha, \beta, \gamma, s_1, s_2)$
\\~\\
Using the similar reasoning as in $(xiii)$.
\\~\\
$(xvi)$ $\tilde{B}_1(\alpha, \beta, \gamma) \approx \tilde{B}_{21}(\alpha, \beta, \gamma, s_1, s_2)$
\\~\\
By using the equivalence $B_{1}(\alpha, \beta) \approx B_{14}(\alpha, \beta, s)$ from Theorem A  and in view of Theorem \ref{IThA} (vii)
\begin{align*} 
\sup_{x<y<b}F^\alpha(y)G^\beta(y)
&\approx \inf_{h_1>0} \sup_{x<y<b}h_1^{-s_1}(y) \left(\int_x^y f(z)\Big(h_1(z)+G(z)\Big)^{\frac{\beta+s_1}{\alpha}}dz\right)^{\alpha}
\end{align*}
and also the fact that $H^{\gamma}(x) \geq 0$ as well as $h_1^{-s_1}(x) \geq 0 $, we get
\begin{align*}
\tilde{B}_1(\alpha, \beta, \gamma)=&\sup_{a<x<b}\tilde{B}_1(x; \alpha, \beta, \gamma)\\
=&\sup_{a<x<b}F^{\alpha}(x)G^{\beta}(x)H^{\gamma}(x) \\
=&\sup_{a<x<b}H^{\gamma}(x)\sup_{x<y<b}F^{\alpha}(y)G^{\beta}(y) \\
\approx &\sup_{a<x<b}H^{\gamma}(x)\inf_{h_1 \geq 0}\sup_{x<y<b}\left(\int_{x}^{y}f(t)\Big(h_1(t)+G(t)\Big)^{\frac{\beta+s_1}{\alpha}}\, dt\right)^{\alpha}h_1^{-s_1}(y) \\
=&\inf_{h_1 \geq 0}\sup_{a<x<b}H^{\gamma}(x)\sup_{x<y<b}h_1^{-s_1}(y)\left(\int_{x}^{y}f(t)\Big(h_1(t)+G(t)\Big)^{\frac{\beta+s_1}{\alpha}}\, dt\right)^{\alpha} \\
\leq &\inf_{h_1 \geq 0}\sup_{a<y<b}\left\{ \sup_{y<x<b}h_1^{-s_1}(x)\right\} \left\{\sup_{a<x<y}\left(\int_{x}^{y}f(t)\Big(h_1(t)+G(t)\Big)^{\frac{\beta+s_1}{\alpha}}\, dt\right)^{\alpha} H^{\gamma}(x)\right\} \\
\approx &\inf_{h_1 \geq 0}\sup_{a<y<b} \left\{ \sup_{y<x<b}h_1^{-s_1}(x)\right\} \, \times  \\
& \hskip 2em  \times \Bigg\{\inf_{h_2 \geq 0}\sup_{a<x<y}h_2^{-s_2}(x) \left(\int_{a}^{x}f(t)\Big(h_1(t)+G(t)\Big)^{\frac{\beta+s_1}{\alpha}}\Big(h_2(t)+H(t)\Big)^{\frac{\gamma+s_2}{\alpha}}\, dt\right)^{\alpha} \Bigg\}\\
=&\inf_{h_1 \geq 0}\inf_{h_2 \geq 0}\sup_{a<y<b} \left\{ \sup_{y<x<b}h_1^{-s_1}(x)\right\}\, \times  \\
& \hskip 3em \times  \Bigg\{\sup_{a<x<y}h_2^{-s_2}(x) \left(\int_{a}^{x}f(t)\Big(h_1(t)+G(t)\Big)^{\frac{\beta+s_1}{\alpha}}\Big(h_2(t)+H(t)\Big)^{\frac{\gamma+s_2}{\alpha}}\, dt\right)^{\alpha} \Bigg\} \\
=&\inf_{h_1, h_2 \geq 0}\sup_{a<x<b}\left(\int_{a}^{x}f(t)\Big(h_1(t)+G(t)\Big)^{\frac{\beta+s_1}{\alpha}}\Big(h_2(t)+H(t)\Big)^{\frac{\gamma+s_2}{\alpha}}\, dt\right)^{\alpha}\, \times \\
& \hskip 20em \times \left\{ \sup_{x<y<b}h_1^{-s_1}(y)\right\} h_2^{-s_2}(x) \\
=&\inf_{h_1, h_2 \geq 0}\sup_{a<x<b} \tilde{B}_{21}(x; \alpha, \beta, \gamma, s_1, s_2; h_1, h_2) \\
=&\tilde{B}_{21}(\alpha, \beta, \gamma, s_1, s_2).
\end{align*}
Thus
\begin{align} \label{e2.4}
\tilde{B}_1(\alpha, \beta, \gamma) \lesssim \tilde{B}_{21}(\alpha, \beta, \gamma, s_1, s_2).
\end{align}
Since for $h_1(x)= G(x)$ and $h_2(x)=H(x)$, we have
\begin{align*}
\tilde{B}_{21}(x; \alpha, \beta, \gamma, s_1, s_2; G, H)=&2^{\beta+\gamma+s_1+s_2}\left(\int_{a}^{x}f(t)G^{\frac{\beta+s_1}{\alpha}}(t)H^{\frac{\gamma+s_2}{\alpha}}(t)\, dt\right)^{\alpha} G^{-s_1}(x) H^{-s_2}(x) \\
=& 2^{\beta+\gamma+s_1+s_2} \tilde{B}_{3}(x; \alpha, \beta, \gamma, s_1, s_2) \\
\leq & 2^{\beta+\gamma+s_1+s_2} \sup_{a<x <b}\tilde{B}_{3}(x; \alpha, \beta, \gamma, s_1, s_2) \\
=& 2^{\beta+\gamma+s_1+s_2} \tilde{B}_{3}(\alpha, \beta, \gamma, s_1, s_2) \\
\approx &  2^{\beta+\gamma+s_1+s_2} \tilde{B}_{1}(\alpha, \beta, \gamma).
\end{align*}
The last relation holds by using $\tilde{B}_{1}(\alpha, \beta, \gamma) \approx \tilde{B}_{3}(\alpha, \beta, \gamma, s_1, s_2)$. Clearly we get that
\begin{align} \label{e2.5}
\tilde{B}_{21}(\alpha, \beta, \gamma, s_1, s_2) \lesssim \tilde{B}_{1}(\alpha, \beta, \gamma).
\end{align}
Therefore, in view of \eqref{e2.4} and \eqref{e2.5}, we find
\begin{align*}
\tilde{B}_1(\alpha, \beta, \gamma) \approx \tilde{B}_{21}(\alpha, \beta, \gamma, s_1, s_2).
\end{align*}
\\~\\
$(xvii)$ $\tilde{B}_1(\alpha, \beta, \gamma) \approx  \tilde{B}_{22}(\alpha, \beta, \gamma, s_1, s_2)$
\\~\\
Using the similar reasoning as in $(xvi)$.
\\~\\
$(xviii)$ $\tilde{B}_1(\alpha, \beta, \gamma) \approx  \tilde{B}_{23}(\alpha, \beta, \gamma, s_1, s_2)$
\\~\\
Using the similar reasoning as in $(xvi)$.
\end{proof}

We apply Theorem \ref{t2.1} and provide equivalent conditions for the bilinear inequality \eqref{H2} to hold.

\begin{theorem}\label{t2.2}
Let $1< \max(p_1,p_2) \leq q < \infty$ with $s,s_1,s_2$ be positive numbers. Define $U$ as in \eqref{e1} and
\begin{equation*}
V_i(x)= \int_{a}^{x}v_i^{1-p_i'}(t)\, dt, \quad i=1,2.
\end{equation*}
Moreover define
\begin{align*}
&\tilde{A}_{1}(s_{1},s_{2}):= {\displaystyle \sup_{a<x<b}\left({\int_{x}^{b}u(t)V_1^{\left(\frac{\frac{1}{p_1'}-s_{1}}{\frac{1}{q}}\right)}(t)V_2^{\left(\frac{\frac{1}{p_2'}-s_{2}}{\frac{1}{q}}\right)}(t)\, dt}\right)^{\frac{1}{q}}V_1^{s_{1}}(x)V_2^{s_{2}}(x)};\\
&\tilde{A}_{2}(s_{1},s_{2}):=\sup_{a<x<b} \left({\int_{a}^{x}u(t)V_1^{\left(\frac{\frac{1}{p_1'}+s_1}{\frac{1}{q}}\right)}(t)V_2^{\left(\frac{\frac{1}{p_2'}+s_2}{\frac{1}{q}}\right)}(t)\, dt}\right)^{\frac{1}{q}}V_1^{-s_{1}}(x)V_2^{-s_{2}}(x); \\
& \tilde{A}_{3}(s_1, s_2):= \sup_{a<x<b} \left({\int_{a}^{x}u(t)V_1^{\left(\frac{\frac{1}{p_1'}}{\frac{1}{q}+s_1}\right)}(t)V_2^{\left(\frac{\frac{1}{p_2'}-s_2}{\frac{1}{q}+s_1}\right)}(t)\, dt}\right)^{\frac{1}{q}+s_1}U^{-s_{1}}(x)V_2^{s_{2}}(x); \\
& \tilde{A}_{4}(s_1, s_2):= \sup_{a<x<b} \left({\int_{a}^{x}u(t)V_1^{\left(\frac{\frac{1}{p_1'}-s_2}{\frac{1}{q}+s_1}\right)}(t)V_2^{\left(\frac{\frac{1}{p_2'}}{\frac{1}{q}+s_1}\right)}(t)\, dt}\right)^{\frac{1}{q}+s_1}U^{-s_{1}}(x)V_1^{s_{2}}(x); \\
& \tilde{A}_{5}(s_1, s_2):= \sup_{a<x<b} \left({\int_{a}^{x}u(t)V_1^{\left(\frac{\frac{1}{p_1'}}{\frac{1}{q}-s_1}\right)}(t)V_2^{\left(\frac{\frac{1}{p_2'}+s_2}{\frac{1}{q}-s_1}\right)}(t)\, dt}\right)^{\frac{1}{q}-s_1}U^{s_{1}}(x)V_2^{-s_{2}}(x), \quad \frac{1}{q} > s_1; \\
& \tilde{A}_{6}(s_1, s_2):= \sup_{a<x<b} \left({\int_{a}^{x}u(t)V_1^{\left(\frac{\frac{1}{p_1'}+s_2}{\frac{1}{q}-s_1}\right)}(t)V_2^{\left(\frac{\frac{1}{p_2'}}{\frac{1}{q}-s_1}\right)}(t)\, dt}\right)^{\frac{1}{q}-s_1}U^{s_{1}}(x)V_1^{-s_{2}}(x), \quad \frac{1}{q} > s_1; \\
& \tilde{A}_{7}(s_1, s_2):= \sup_{a<x<b} \left({\int_{x}^{b}u(t)V_1^{\left(\frac{\frac{1}{p_1'}}{\frac{1}{q}-s_1}\right)}(t) \, dt}\right)^{\frac{1}{q}-s_1}V_2^{s_{2}}(x)\left({\int_{x}^{b}u(t)V_2^{\left(\frac{\frac{1}{p_2'}-s_2}{s_1}\right)}(t)\, dt}\right)^{s_1}, \quad \frac{1}{q} < s_1; \\
& \tilde{A}_{8}(s_1, s_2):= \sup_{a<x<b} \left({\int_{x}^{b}u(t)V_2^{\left(\frac{\frac{1}{p_2'}}{\frac{1}{q}-s_1}\right)}(t) \, dt}\right)^{\frac{1}{q}-s_1}V_1^{s_{2}}(x)\left({\int_{x}^{b}u(t)V_1^{\left(\frac{\frac{1}{p_1'}-s_2}{s_1}\right)}(t)\, dt}\right)^{s_1}, \quad \frac{1}{q} < s_1; \\
&\tilde{A}_{9}(s_1, s_2):= \sup_{a<x<b} \left(\int_{a}^{x}v_1^{1-p_1'}(t)U^{\left(\frac{\frac{1}{2q}}{\frac{1}{p_1'}+s_1}\right)}(t)\, dt\right)^{\frac{1}{p_1'}+s_1}\left(\int_{a}^{x}v_2^{1-p_2'}(t)U^{\left(\frac{\frac{1}{2q}}{\frac{1}{p_2'}+s_2}\right)}(t)\, dt\right)^{\frac{1}{p_2'}+s_2}\, \times \\
& \hskip 30em \times V_1^{-s_1}(x)V_2^{-s_2}(x); \\
& \tilde{A}_{10}(s):= \sup_{a<x<b} \left(\int_{x}^{b}u(t) V_1^{\left(\frac{\frac{1}{p_1'}(1-s)}{\frac{1}{q}}\right)}(t)V_2^{\left(\frac{\frac{1}{p_2'}(1-s)}{\frac{1}{q}}\right)}(t)\, dt\right)^{\frac{1}{q}}V_1^{\frac{s}{p_1'}}(x)V_2^{\frac{s}{p_2'}}(x); \\
& \tilde{A}_{11}(s):= \sup_{a<x<b} \left(\int_{a}^{x}u(t) V_1^{\left(\frac{\frac{1}{p_1'}(1+s)}{\frac{1}{q}}\right)}(t)V_2^{\left(\frac{\frac{1}{p_2'}(1+s)}{\frac{1}{q}}\right)}(t)\, dt\right)^{\frac{1}{q}}V_1^{-\frac{s}{p_1'}}(x)V_2^{-\frac{s}{p_2'}}(x); \\
&\tilde{A}_{12}(s):= \sup_{a<x<b} \left(\int_{x}^{b}u(t)V_1^{\left(\frac{\frac{1}{p_1'}}{\frac{1}{q}+s}\right)}(t)V_2^{\left(\frac{\frac{1}{p_2'}}{\frac{1}{q}+s}\right)}(t)\, dt\right)^{\frac{1}{q}+s}U^{-s}(x); \\
&\tilde{A}_{13}(s):= \sup_{a<x<b} \left(\int_{a}^{x}u(t)V_1^{\left(\frac{\frac{1}{p_1'}}{\frac{1}{q}-s}\right)}(t)V_2^{\left(\frac{\frac{1}{p_2'}}{\frac{1}{q}-s}\right)}(t)\, dt\right)^{\frac{1}{q}-s}U^{s}(x), \quad \frac{1}{q} > s; \\
&\tilde{A}_{14}(s):= \sup_{a<x<b} \left(\int_{x}^{b}u(t)V_1^{\left(\frac{\frac{1}{p_1'}}{\frac{1}{q}-s}\right)}(t)V_2^{\left(\frac{\frac{1}{p_2'}}{\frac{1}{q}-s}\right)}(t)\, dt\right)^{\frac{1}{q}-s}U^{s}(x), \quad \frac{1}{q} < s; \\
&\tilde{A}_{15}(s):= \inf_{h_1 \geq 0} \sup_{a<x<b} \left(\int_{x}^{b}u(t)h_1^{\left(\frac{\frac{1}{p_2'}-s}{\frac{1}{q}}\right)}(t)\, dt\right)^{\frac{1}{q}} V_1^{\frac{1}{p_1'}}(x)\Bigg(h_1(x)+ V_2(x)\Bigg)^s, \quad \frac{1}{p_2'} < s; \\
&\tilde{A}_{16}(s):=  \inf_{h_1 \geq 0} \sup_{a<x<b} \left(\int_{x}^{b} u(t)h_1^{\left(\frac{\frac{1}{p_1'}-s}{\frac{1}{q}}\right)}(t)\, dt\right)^{\frac{1}{q}} V_2^{\frac{1}{p_2'}}(x)\Bigg(h_1(x)+ V_1(x)\Bigg)^s, \quad \frac{1}{p_1'} < s; \\
& \tilde{A}_{17}(s_1, s_2):= \inf_{h_1, h_2 \geq 0} \sup_{a<x<b} \left(\int_{x}^{b} u(t)h_1^{\left(\frac{\frac{1}{p_1'}-s_1}{\frac{1}{q}}\right)}(t) h_2^{\left(\frac{\frac{1}{p_2'}-s_2}{\frac{1}{q}}\right)}(t) \, dt\right)^{\frac{1}{q}} \times \\
& \hskip 9em \times \left\{\sup_{a<y<x} \Bigg(h_1(y)+ V_1(y)\Bigg)^{s_1}\right\} \Bigg(h_2(x)+ V_2(x)\Bigg)^{s_2}, \quad \frac{1}{p_1'}<s_1,\, \frac{1}{p_2'}<s_2; \\
& \tilde{A}_{18}(s_1, s_2):= \inf_{h_1, h_2 \geq 0} \sup_{a<x<b} \left(\int_{x}^{b} u(t)h_1^{\left(\frac{\frac{1}{p_1'}-s_1}{\frac{1}{q}}\right)}(t) h_2^{\left(\frac{\frac{1}{p_2'}-s_2}{\frac{1}{q}}\right)}(t) \, dt\right)^{\frac{1}{q}} \times \\
& \hskip 9em \times \Bigg(h_1(x)+ V_1(x)\Bigg)^{s_1} \left\{\sup_{a<y<x}\Bigg(h_2(y)+ V_2(y)\Bigg)^{s_2}\right\}, \quad \frac{1}{p_1'}<s_1,\, \frac{1}{p_2'}<s_2; \\
& \tilde{A}_{19}(s_1, s_2):= \inf_{h_1, h_2 \geq 0} \sup_{a<x<b} \left(\int_{x}^{b} u(t)h_1^{\left(\frac{\frac{1}{p_1'}-s_1}{\frac{1}{q}}\right)}(t) h_2^{\left(\frac{\frac{1}{p_2'}-s_2}{\frac{1}{q}}\right)}(t) \, dt\right)^{\frac{1}{q}} \times \\
& \hskip 6em \times \left\{\sup_{a<y<x} \Bigg(h_1(y)+ V_1(y)\Bigg)^{s_1}\right\} \left\{\sup_{a<y<x}\Bigg(h_2(y)+ V_2(y)\Bigg)^{s_2}\right\}, \quad \frac{1}{p_1'}<s_1,\, \frac{1}{p_2'}<s_2; \\
&\tilde{A}_{20}(s_1, s_2):=  \inf_{h_1, h_2 \geq 0} \sup_{a<x<b} \left(\int_{a}^{x}u(t)\Bigg(h_1(t)+V_1(t)\Bigg)^{\left(\frac{\frac{1}{p_1'}+s_1}{\frac{1}{q}}\right)}\Bigg(h_2(t)+V_2(t)\Bigg)^{\left(\frac{\frac{1}{p_2'}+s_2}{\frac{1}{q}}\right)}\, dt\right)^{\frac{1}{q}}\, \times  \\
& \hskip 25em  \times \left\{\sup_{x<y<b}h_1^{-s_1}(y)\right\}h_2^{-s_2}(x); \\
&\tilde{A}_{21}(s_1, s_2):=  \inf_{h_1, h_2 \geq 0} \sup_{a<x<b} \left(\int_{a}^{x}u(t)\Bigg(h_1(t)+V_1(t)\Bigg)^{\left(\frac{\frac{1}{p_1'}+s_1}{\frac{1}{q}}\right)}\Bigg(h_2(t)+V_2(t)\Bigg)^{\left(\frac{\frac{1}{p_2'}+s_2}{\frac{1}{q}}\right)}\, dt\right)^{\frac{1}{q}}\, \times  \\
& \hskip 25em  \times h_1^{-s_1}(x) \left\{\sup_{x<y<b}h_2^{-s_2}(y)\right\}; \\
&\tilde{A}_{22}(s_1, s_2):=  \inf_{h_1, h_2 \geq 0} \sup_{a<x<b} \left(\int_{a}^{x}u(t)\Bigg(h_1(t)+V_1(t)\Bigg)^{\left(\frac{\frac{1}{p_1'}+s_1}{\frac{1}{q}}\right)}\Bigg(h_2(t)+V_2(t)\Bigg)^{\left(\frac{\frac{1}{p_2'}+s_2}{\frac{1}{q}}\right)}\, dt\right)^{\frac{1}{q}}\, \times  \\
& \hskip 25em  \times \left\{\sup_{x<y<b}h_1^{-s_1}(y)\right\} \left\{\sup_{x<y<b}h_2^{-s_2}(y)\right\}.
\end{align*}
Then the inequality \eqref{H2} holds for all $f,g \in \mathfrak{M}^{+}$ if and only if any of the $\tilde{A}_i(s_1,s_2)$,$\tilde{A}_j(s),\, (i=1, 2,..., 9, 17, 18,..., 22),\, (j=10, 11,..., 16)$ is finite. Also the best constant $C$ in \eqref{H2} satisfies $C \approx \tilde{A}_i(s_1,s_2) \approx \tilde{A}_j(s), \, (i=1, 2,..., 9, 17, 18,..., 22), \, (j=10, 11,..., 16)$.
\end{theorem}

\begin{proof}
By putting $f(x)=u(x)$,$g(x)=v_1^{1-p_1'}(x)$ in \eqref{e1.4} and $h(x)=v_2^{1-p_2'}(x)$ in \eqref{e2.1}, we get
\begin{equation*}
F(x)=U(x),\, G(x)=V_1(x), \,H(x)=V_2(x).
\end{equation*}
Consequently for $\alpha=\frac{1}{q}$, $\beta=\frac{1}{p_1'}$ and $\gamma=\frac{1}{p_2'}$, we have
\begin{align*}
&\tilde{A}_i(s_1,s_2)=\tilde{B}_{i+1}\left(\frac{1}{q},\frac{1}{p_1'},\frac{1}{p_2'},s_1,s_2\right); \quad (i=1, 2,..., 9, 17, 18,..., 22), \\
&\tilde{A}_j(s)=\tilde{B}_{j+1}\left(\frac{1}{q},\frac{1}{p_1'},\frac{1}{p_2'},s\right); \quad (j=10, 11,..., 16).
\end{align*}
Now by Theorem \ref{t2.1}, these are all equivalent to $\tilde{B}_1\left(\frac{1}{q},\frac{1}{p_1'},\frac{1}{p_2'}\right)\overset{\eqref{BHC}}=\mathcal{D}$and $\mathcal{D}$ is necessary and sufficient condition for the inequality \eqref{H2} to hold.

Since the constant $C$ in \eqref{H2} verifies $C \approx \mathcal{D}=\tilde{B}_1\left(\frac{1}{q},\frac{1}{p_1'},\frac{1}{p_2'}\right)$, it follows that $C \approx \tilde{A}_i(s_1,s_2) \approx \tilde{A}_j(s),\, (i=1, 2,..., 9, 17, 18,..., 22), \, (j=10, 11,..., 16)$.
\end{proof}

\section{Bilinear Type Geometric Mean Operator}
Consider the Hardy averaging operator $A:$
\begin{align*}
(Af)(x):= \frac{1}{x}\int_{0}^{x}f(t)\, dt
\end{align*}
and the geometric mean operator $T$ which is defined by
\begin{align*}
(Tf)(x):= \exp\left(\frac{1}{x}\int_{0}^{x} \ln(f(t))\, dt\right), \quad f \in \mathfrak{M}^{+}.
\end{align*}
Note that
\begin{equation}
\label{LHA}
\lim_{\alpha\to 0}\Big\{(Af^{\alpha})(x)\Big\}^{\frac{1}{\alpha}}=Tf(x).
\end{equation}
Therefore, it seems reasonable to study the inequality
\begin{equation}
\label{GM}
\left(\int_{0}^{\infty}\left(Tf\right)^q(x)u(x)\, dx\right)^{\frac{1}{q}} \leq C\left(\int_{0}^{\infty}f^p(x)v(x)\, dx\right)^{\frac{1}{p}}
\end{equation}
via \eqref{H1} with suitable modifications. However, it was observed in \cite{pla1}, \cite{pla2} that the Muckenhoupt condition \eqref{H2} is not suitable for these modifications. To this end, Persson and Stepanov \cite{PS} obtained an alternate condition for the inequality \eqref{H1} and by suitable modifications studied the inequality \eqref{GM}. Precisely, they proved the following
\medskip

\textbf{Theorem C.} \emph{Let $0<p \leq q <\infty$. Then the inequality \eqref{GM} with the best constant $C$ holds for $f \in \mathfrak{M}^{+}$ if and only if following holds}
\begin{equation}\label{GMC}
\mathcal{B}:= \sup_{t>0}t^{-\frac{1}{p}}\left(\int_{0}^{t}\left[T\left(\frac{1}{v(x)}\right)\right]^{\frac{q}{p}}u(x)\, dx\right)^{\frac{1}{q}} < \infty.
\end{equation}
\medskip

This was perhaps one of the motivations to obtain several equivalent conditions for the inequality \eqref{H1}. Let us mention that the inequality \eqref{GM} has also independently been studied in \cite{heinig2}, \cite{opic}, \cite{pick}.

Our aim, in this section, is to study the following bilinear type inequality involving the geometric mean operator $T$:
\begin{align}
\label{BGM}
\left(\int_{0}^{\infty}\left(Tf\right)^q(x)\left(Tg\right)^q(x)u(x)\, dx\right)^{\frac{1}{q}} &\leq C\left(\int_{0}^{\infty}f^{p_1}(x)v_1(x)\, dx\right)^{\frac{1}{p_1}}\nonumber\\
&\qquad\times\left(\int_{0}^{\infty}f^{p_2}(x)v_2(x)\, dx\right)^{\frac{1}{p_2}}.
\end{align}

We prove the following:

\begin{theorem}\label{t3.1}
Let $1< \max(p_1,p_2)<q < \infty$.Then the inequality \eqref{BGM} holds for all $f, g \, \in \mathfrak{M}^{+}$ if and only if
\begin{align*}
\tilde{\mathcal{B}}:= \sup_{t > 0}t^{-\left(\frac{1}{p_1}+\frac{1}{p_2}\right)}\left(\int_{0}^{t}\left[T\left(\frac{1}{v_1(x)}\right)\right]^{\frac{q}{p_1}}\left[T\left(\frac{1}{v_2(x)}\right)\right]^{\frac{q}{p_2}}u(x)\, dx\right)^{\frac{1}{q}}<\infty.
\end{align*}
\end{theorem}

\begin{proof}
	By using \eqref{GMC} from Theorem C, we get
	\begin{align*}
	C&= \sup_{g \in \mathfrak{M}^{+}}\sup_{f \in \mathfrak{M}^{+}}\frac{\left(\int_{0}^{\infty}\left(Gf\right)^q(x)\left(Gg\right)^q(x)u(x)\, dx\right)^{\frac{1}{q}}}{\Vert f \Vert_{p_1,v_1}\Vert g \Vert_{p_2,v_2}} \\
	&\approx \sup_{g \in \mathfrak{M}^{+}}\Vert g \Vert_{p_2,v_2}^{-1}\sup_{t_1 > 0}t_1^{-\frac{1}{p_1}}\left(\int_{0}^{t_1}\left[G\left(\frac{1}{v_1(x)}\right)\right]^{\frac{q}{p_1}}(Gg)^q(x)u(x)\, dx\right)^{\frac{1}{q}} \\
	&=\sup_{t_1 > 0}t_1^{-\frac{1}{p_1}}\sup_{g \in \mathfrak{M}^{+}}\Vert g \Vert_{p_2,v_2}^{-1}\left(\int_{0}^{t_1}\left[G\left(\frac{1}{v_1(x)}\right)\right]^{\frac{q}{p_1}}(Gg)^q(x)u(x)\, dx\right)^{\frac{1}{q}} \\
	&\approx \sup_{t_1 > 0}t_1^{-\frac{1}{p_1}}\sup_{0<t_2<t_1}t_2^{-\frac{1}{p_2}}\left(\int_{0}^{t_2}\left[G\left(\frac{1}{v_1(x)}\right)\right]^{\frac{q}{p_1}}\left[G\left(\frac{1}{v_2(x)}\right)\right]^{\frac{q}{p_2}}u(x)\, dx\right)^{\frac{1}{q}} \\
	&=\sup_{t > 0}t^{-\left(\frac{1}{p_1}+\frac{1}{p_2}\right)}\left(\int_{0}^{t}\left[G\left(\frac{1}{v_1(x)}\right)\right]^{\frac{q}{p_1}}\left[G\left(\frac{1}{v_2(x)}\right)\right]^{\frac{q}{p_2}}u(x)\, dx\right)^{\frac{1}{q}} \\
	&=\tilde{\mathcal{B}}.
	\end{align*}
	
\end{proof}

\section{Acknowledgement}
The work of first author,  has been partially supported by Shota Rustaveli National Science Foundation of Georgia (SRNSFG) [grant number FR17-589], by Grant 18-00580S of the Czech Science Foundation and RVO:67985840.

The second author acknowledges the MATRICS Research Grant No. MTR/2017/000126 of SERB, Department of Science and Technology, India.

\bigskip


\noindent {\it Amiran Gogatishvili} \\
Institute of Mathematics\\
of the Czech Academy of Sciences \\
 \v Zitn\' a 25, 115 67 Prague 1\\
Czech Republic\\
Email: gogatish@math.cas.cz
\bigskip

\noindent {\it Pankaj Jain} \\
Department of Mathematics\\
South Asian University\\
Chanakya Puri, New Delhi-110021\\
India\\
Email: pankaj.jain@sau.ac.in \&  pankajkrjain@hotmail.com
\bigskip

\noindent {\it Saikat Kanjilal} \\
Department of Mathematics\\
South Asian University\\
Chanakya Puri, New Delhi-110021\\
India\\
Email: saikatkanjilal@students.sau.ac.in \&  saikat.kanjilal.07@gmail.com


\begin{thebibliography}{99}

\bibitem{Agui}M.I. Aguilar Ca\~{n}estro, P.Ortega Salvador and C.Ram\'{i}rez Torreblanca, \textit{Weighted bilinear Hardy inequalities}, J. Math. Anal. Appl., 387 (2012), 320--334.

\bibitem{GKPW}A. Gogatishvili, A. Kufner, L.-E. Persson, and A. Wedestig, \textit{An equivalence theorem for some scales of integral conditions related to Hardy's inequality with applications}, Real Anal. Exchange, 29 (2004), 867--880.

\bibitem{GKP}A. Gogatishvili, A. Kufner and L.E. Persson, \textit{Some new scales of characterization of Hardy's inequality}, Proc. Est. Acad. Sci., 59 (2010), 7--18.

\bibitem{heinig2}H.P. Heinig, R. Kerman and M. Krbec, \textit{Weighted exponential inequalities}, Georgian Math. J., 8 (2001), 69--86.

\bibitem{pla1}  P. Jain, L.-E. Persson and A. Wedestig, \textit{From Hardy to
	Carleman and general mean-type inequalities, }Function Spaces and
Applications, CRC Press (New York)/Narosa Publishing House (New Delhi)/Alpha
Science (Pangbourne) (2000), 117--130 .

\bibitem{pla2}  P. Jain, L.-E. Persson and A. Wedestig, \textit{Carleman-Knopp type inequalities via Hardy inequalities}, Math. Inequal. Appl., 4 (2001), 343--355.


\bibitem{KPS}S. Kanjilal, L.-E. Persson and G.E. Shambilova, \textit{Equivalent integral conditions related to bilinear Hardy-type inequalities }, Math. Inequal. Appl., 22 (2019), 1535--1548.

\bibitem{Kr1}M. K\v{r}epela, \textit{Iterating bilinear Hardy inequalities}, Proc. Edinburgh Math. Soc., 60 (2017), 955--971.

\bibitem{Mucken}B. Muckenhoupt, \textit{Hardy's inequality with weights}, Studia Math., 44 (1972), 31--38.

\bibitem{opic}  B. Opic and P. Gurka, \textit{Weighted inequalities for geometric means}, Proc. Amer. Math. Soc., 3 (1994), 771--779.

\bibitem{PS} L.-E. Persson, V.D. Stepanov, \textit{Weighted integral inequalities with the geometric mean operator}, J. Inequal. Appl., 7 (2002), 727--746.

\bibitem{PSW} L.-E. Persson, V.D. Stepanov, P. Wall,  \textit{Some scale of equivalent weight characterizations of Hardy's  inequality: the case $q<p$}, Math. Inequal. Appl., 10 (2007), 267--279.

\bibitem{pick}  L. Pick and B. Opic, \textit{On the geometric mean operator}, J. Math. Anal. Appl., 183 (1994), 652--662.


\end{thebibliography}
\end{document}